\documentclass[a4paper, 11pt]{article}
\usepackage{latexsym,amsfonts, amsmath, amssymb, amsthm}
\usepackage[utf8]{inputenc}
\usepackage[final=true,protrusion=true,expansion=true]{microtype}
\usepackage[noadjust]{cite}
\usepackage{enumitem}
\usepackage{color}
\usepackage{mathtools}

\usepackage{booktabs}
\usepackage{multirow}

\usepackage{subcaption}
\usepackage[font=scriptsize,labelfont=bf]{caption}
\usepackage[affil-it]{authblk}

\usepackage{geometry}
\usepackage{hyperref}
\hypersetup{
	final,
    colorlinks=true,
    linkcolor=blue,
    filecolor=magenta,
    urlcolor=cyan,
}

\theoremstyle{plain}
\newtheorem{theorem}{Theorem}
\newtheorem{proposition}[theorem]{Proposition}
\newtheorem{lemma}[theorem]{Lemma}

\theoremstyle{plain}
\newtheorem{definition}[theorem]{Definition}

\newtheorem{remark}[theorem]{Remark}

\providecommand{\Supp}{\operatorname{supp}}                            % support
\providecommand{\supp}{\Supp}
\providecommand{\Id}{\Op{Id}}                     % Identity operator

\providecommand{\CC}{{\cal C}}
\providecommand{\CE}{{\cal E}}

\providecommand{\CN}{{\cal N}}

\providecommand{\CP}{{\cal P}}

\providecommand{\bbE}{\mathbb{E}}

\providecommand{\bbR}{\mathbb{R}}

\providecommand*{\abs}[1]{\left|{#1}\right|} % Double bar norm
\providecommand*{\N}[1]{\left\|{#1}\right\|} % Double bar norm
\providecommand*{\Nnormal}[1]{\|{#1}\|} % Double bar norm
\providecommand*{\abs}[1]{\left|{#1}\right|} % absolute value, same as \SN
\newcommand*{\Op}[1]{\mathsf{#1}} % Operators
\usepackage{amssymb,amsmath,amscd,amsfonts,amsthm,bbm,mathrsfs,yhmath}
\usepackage{hyperref}
\usepackage{url}
\usepackage[utf8]{inputenc} 
\usepackage{cleveref}
\usepackage{algorithm}
\usepackage{algpseudocode}
\usepackage{graphicx}
\usepackage{color}
\usepackage{listings}
\usepackage{physics}
\usepackage{nicefrac}
\usepackage{subcaption}

\usepackage{booktabs}
\usepackage[svgnames,table]{xcolor}
\usepackage{stackengine}

\usepackage{geometry}
\usepackage{fullpage}
\newcommand{\ff}{G_W}

\newcommand{\bV}{\overbar{V}}

\newcommand{\gC}{C_{\revised{\kappa}}}

\newcommand{\cin}{C_{2}}
\newcommand{\cfm}{C_{3}}
\newcommand{\vrho}{v_{\alpha}(\rho_t)}

\newcommand{\rhoM}{\mathcal{P}_{v_b,R}\left(\vrho\right)}
\newcommand{\rhoMM}{\mathcal{P}_{v_b,R}\left(v_{\alpha}(\widehat{\rho}_t^N)\right)}

\newcommand{\normmsq}[1]{\left\| #1 \right\|_2^2}
\newcommand{\normm}[1]{\left\| #1 \right\|_2}
\newcommand{\inner}[2]{\left< #1 , #2 \right>}

\newcommand{\Ep}[1]{\mathbb{E}\left[#1\right]}

\newcommand{\Gwn}{G_{W,N,\revised{\kappa}}(t)}
\newcommand{\Gwno}{G_{W,N,\revised{\kappa}}(0)}
\usepackage[colorinlistoftodos,bordercolor=orange,backgroundcolor=orange!20,linecolor=orange,textsize=scriptsize]{todonotes}

\newif\ifrevised
\newcommand{\revised}[1]{%
	\ifrevised
	\color{purple} #1 \color{black} %
	\else
	#1%
	\fi}
\revisedfalse
\newcommand{\overbar}[1]{\makebox[0pt]{$\phantom{#1}\mkern 1.5mu\overline{\mkern-1.5mu\phantom{#1}\mkern-1.5mu}\mkern 1.5mu$}#1}
\renewcommand{\underbar}[1]{\makebox[0pt]{$\phantom{#1}\mkern 1.5mu\underline{\mkern-1.5mu\phantom{#1}\mkern-1.5mu}\mkern 1.5mu$}#1}

\DeclareMathOperator*{\Law}{Law}
\newcommand{\overbarscript}[1]{\mkern 1.5mu\overline{\mkern-1.5mu#1\mkern-1.5mu}\mkern 0mu}

\newcommand{\globmin}{v^*}
\renewcommand{\CE}{f}
\newcommand{\minobj}{\underbar \CE}

\newcommand{\indivmeasure}[0]{\varrho} %\widetilde\rho already used

\newcommand{\empmeasure}[1]{\widehat\rho^N_{#1}}

\newcommand{\omegaa}[0]{\omega_{\alpha}}

\newcommand{\conspoint}[1]{v_{\alpha}({#1})}

\title{\usefont{OT1}{bch}{b}{n}
	\huge Consensus-Based Optimization\\with Truncated Noise\\
}

\date{}

\author[1,2,3]{Massimo Fornasier\thanks{Email: \texttt{massimo.fornasier@cit.tum.de} (corresponding author)}}
\author[4,5,6]{Peter Richt\'{a}rik\thanks{Email: \texttt{peter.richtarik@kaust.edu.sa}}}
\author[1,2]{Konstantin Riedl\thanks{Email: \texttt{konstantin.riedl@ma.tum.de}}}
\author[4,5]{Lukang Sun\thanks{Email: \texttt{lukang.sun@kaust.edu.sa} }}

\affil[1]{Technical University of Munich, School of Computation, Information and Technology, Department of Mathematics, Munich, Germany}
\affil[2]{Munich Center for Machine Learning, Munich, Germany}
\affil[3]{Munich Data Science Institute, Germany}
\affil[4]{King Abdullah University of Science and Technology, Thuwal, Saudi Arabia}
\affil[5]{KAUST AI Initiative, Thuwal, Saudi Arabia}
\affil[6]{SDAIA-KAUST Center of Excellence in Data Science and Artificial Intelligence, Thuwal, Saudi Arabia}

\begin{document}
	\maketitle
	%\tableofcontents

	\begin{abstract}
		\noindent
		Consensus-based optimization (CBO) is a versatile multi-particle metaheuristic optimization method suitable for performing nonconvex and nonsmooth global optimizations in high dimensions. It has proven effective in various applications while at the same time being amenable to a theoretical convergence analysis.
		In this paper, we explore a variant of CBO, which incorporates truncated noise in order to enhance the well-behavedness of the statistics of the law of the dynamics.
		By introducing this additional truncation in the noise term of the CBO dynamics, we achieve that, in contrast to the original version, higher moments of the law of the particle system can be effectively bounded.
		As a result, our proposed variant exhibits enhanced convergence performance, allowing in particular for wider flexibility in choosing the \revised{noise parameter} of the method as we confirm experimentally.
		By analyzing the time-evolution of the Wasserstein-$2$ distance between the empirical measure of the interacting particle system and the global minimizer of the objective function, we rigorously prove convergence in expectation  of the proposed CBO variant requiring only minimal assumptions on the objective function and on the initialization.
		Numerical evidences demonstrate the benefit of truncating the noise in CBO.
	\end{abstract}
	
	

	\section{Introduction}
	The search for a global minimizer~$\globmin$ of a potentially nonconvex and nonsmooth cost function $$\CE:\bbR^d\to\bbR$$ holds significant importance in a variety of applications throughout applied mathematics, science and technology, engineering, and machine learning.
	Historically, a class of methods known as metaheuristics~\cite{back1997handbook,blum2003metaheuristics} has been developed to address this inherently challenging and, in general, NP-hard problem.
	Examples of such include
	evolutionary programming~\cite{fogel2006evolutionary},
	genetic algorithms~\cite{holland1992adaptation},
	particle swarm optimization (PSO)~\cite{kennedy1995particle},
	simulated annealing~\cite{aarts1989simulated}, and many others.
	These methods work combining local improvement procedures and global strategies by orchestrating deterministic and stochastic advances, with the aim of creating a method capable of robustly and efficiently finding the globally minimizing argument~$\globmin$ of $\CE$.
	However, despite their empirical success and widespread adoption in practice, most metaheuristics lack a solid mathematical foundation that could guarantee their robust convergence to global minimizers under reasonable assumptions.
	
	Motivated by the urge to devise algorithms which converge provably, a novel class of metaheuristics, so-called consensus-based optimization (CBO), originally proposed by the authors of \cite{pinnau2017consensus}, has recently emerged in the literature.
	Due to the inherent simplicity in the design of CBO,
	this class of optimization algorithms lends itself to a rigorous theoretical analysis, as demonstrated in particular in the works \cite{carrillo2018analytical,carrillo2019consensus,ha2021convergence,ha2020convergenceHD,ko2022convergence,fornasier2021consensus,fornasier2021convergence}.
	However, this recent line of research does not just offer a promising avenue for establishing a thorough mathematical framework for understanding the numerically observed successes of CBO methods~\cite{carrillo2019consensus,fornasier2021convergence,fornasier2020consensus_sphere_convergence,riedl2022leveraging,carrillo2023fedcbo}, but beyond that allows to explain the effective use of conceptually similar and wide-spread methods such as PSO as well as at first glance completely different optimization algorithms such as stochastic gradient descent (SGD).
	While the first connection is to be expected and by now made fairly rigorous~\cite{grassi2020particle,cipriani2021zero,qiu2022PSOconvergence} due to CBO indisputably taking PSO as inspiration, the second observation is somewhat surprising, as it builds a bridge between derivative-free metaheuristics and gradient-based learning algorithms.
	Despite CBO solely relying on evaluations of the objective function, recent work~\cite{riedl2023gradient} reveals an intrinsic SGD-like behavior of CBO itself by interpreting it as a certain stochastic relaxation of gradient descent, which provably overcomes energy barriers of nonconvex function.
	These perspectives, and, in particular the already well-investigated convergence behavior of standard CBO, encourage the exploration of improvements to the method in order to allow  overcoming the limitations of traditional metaheuristics mentioned at the start.
	For recent surveys on CBO we refer to \cite{grassi2021mean,totzeck2021trends}.
	
	While the original CBO model~\cite{pinnau2017consensus} has been adapted to solve
	constrained optimizations~\cite{borghi2021constrained,carrillo2021consensus,bae2022constrained}, optimizations on manifolds~\cite{fornasier2020consensus_hypersurface_wellposedness,fornasier2020consensus_sphere_convergence,fornasier2021anisotropic,kim2020stochastic,ha2021emergent},
	multi-objective optimization problems~\cite{borghi2022consensus,borghi2022adaptive,klamroth2022consensus},
	saddle point problems~\cite{huang2022consensus} or the task of sampling~\cite{carrillo2022consensus},
	as well as has been extended to make use of memory mechanisms~\cite{totzeck2020consensus,riedl2022leveraging,borghi2023consensus}, gradient information~\cite{riedl2022leveraging,schillings2022Ensemble}, momentum~\cite{chen2020consensus}, jump-diffusion processes~\cite{kalise2022consensus} or localization kernels for polarization~\cite{bungert2022polarized},
	we focus in this work on a variation of the original model, which incorporates a truncation in the noise term of the dynamics.
	More formally, given a time horizon~$T>0$, a time discretization $t_0 = 0 < \Delta t < \cdots< K \Delta t = t_K = T$ of $[0,T]$, and user-specified parameters~$\alpha,\lambda,\sigma>0$ as well as $v_b,R>0$, we consider the interacting particle system
	\begin{alignat}{2} \label{eq:17}
		\phantom{V_{\revised{{k+1,\Delta t}}}^i - V_{\revised{{k,\Delta t}}}^i} &\begin{aligned}[c]
			\mathllap{V^i_{\revised{{k+1,\Delta t}}} - V^i_{\revised{{k,\Delta t}}}} = &- \Delta t\lambda\left( V^i_{\revised{{k,\Delta t}}} - \mathcal{P}_{v_b,R}\left(\conspoint{\empmeasure{\revised{{k,\Delta t}}}}\right)\right)
			+ \sigma \left(\N{ V^i_{\revised{{k,\Delta t}}}-\conspoint{\empmeasure{\revised{{k,\Delta t}}}}}_2 \wedge M\right) B^i_{\revised{{k,\Delta t}}},\\
		\end{aligned} \\
		\mathllap{V_0^i} &\sim \rho_0 \quad \text{for all } i =1,\ldots,N,
	\end{alignat}
	where $(( B^i_{\revised{{k,\Delta t}}})_{k=0,\ldots,K-1})_{i=1,\ldots,N}$ are independent, identically distributed Gaussian random vectors in $\bbR^d$ with zero mean and covariance matrix $\Delta t \Id_d$.
	Equation~\eqref{eq:17} originates from a simple Euler-Maruyama time discretization~\cite{higham2001algorithmic,platen1999introduction} of the system of stochastic differential equations~(SDEs), expressed in It\^o's form as
	\begin{alignat}{2} \label{eq:5}
		d V^i_t &= -\lambda\left(V^i_t-\rhoMM\right)dt+\sigma\left(\normm{V^i_t-v_{\alpha}(\empmeasure{t})}\wedge M\right)dB^i_t \\
		\mathllap{V_0^i} &\sim \rho_0 \quad \text{for all } i =1,\ldots,N.
	\end{alignat}
	where $((B_t^i)_{t\geq 0})_{i = 1,\ldots,N}$ are now independent standard Brownian motions in $\bbR^d$.
	The empirical measure of the particles at time~$t$ is denoted by 
	$\empmeasure{t} := \frac{1}{N} \sum_{i=1}^{N} \delta_{V_t^i}$.
    Moreover, $\mathcal{P}_{v_b,R}$ is \revised{the projection onto $B_R(v_b)$ defined as}
	\begin{equation}\label{defproj}
		\mathcal{P}_{v_b,R}\left(v\right):=
		\begin{cases}
			v,
			& \text{if $\N{v-v_b}_2\leq R$},\\
			v_b+R\frac{v-v_b}{\N{v-v_b}_2},
			& \text{if $\N{v-v_b}_2>R$}.
		\end{cases}
	\end{equation}
    As a crucial assumption in this paper, the map $\mathcal{P}_{v_b,R}$ depends on $R$ and $v_b$ in such way that $\globmin\in B_R(v_b)$.
	Setting the parameters can be feasible under specific circumstances, as exemplified by the regularized optimization problem $f(v):=\operatorname{Loss}(v)+\revised{\Lambda}\normm{v}$, wherein $\globmin\in B_{\operatorname{Loss}(0)/\revised{\Lambda}}(0)$.
	In the absence of prior knowledge regarding $v_b$ and $R$, a practical approach is to choose $v_b=0$ and assign a sufficiently large value to $R$.
	The first terms in \eqref{eq:17} and \eqref{eq:5}, respectively, impose a deterministic drift of each particle towards the possibly projected momentaneous consensus point $\conspoint{\empmeasure{t}}$, which is a weighted average of the particles' positions and computed according to
	\begin{align} \label{eq:momentaneous_consensus}
		\conspoint{\empmeasure{t}} := \int v \frac{\omegaa(v)}{\N{\omegaa}_{L_1(\empmeasure{t})}}\,d\empmeasure{t}(v).
	\end{align}
	The weights $ \omega_{\alpha}(v):=\exp(-\alpha \CE(v))$ are motivated by the well-known Laplace principle~\cite{dembo2009large}, which states for any absolutely continuous probability distribution $\varrho$ on $\mathbb{R}^d$ that
	\begin{align} \label{eq:laplace_principle}
		\lim\limits_{\alpha\rightarrow \infty}\left(-\frac{1}{\alpha}\log\left(\int\omegaa(v)\,d\indivmeasure(v)\right)\right) = \inf\limits_{v \in \supp(\indivmeasure)}\CE(v)
	\end{align}
	and thus justifies that $\conspoint{\empmeasure{t}}$ serves as a suitable proxy for the global minimizer~$\globmin$ given the currently available information of the particles~$(V^i_t)_{i=1,\dots,N}$.
	The second terms in \eqref{eq:17} and \eqref{eq:5}, respectively, encode the diffusion or exploration mechanism of the algorithm, where, in contrast to standard CBO, we truncate the noise by some fixed constant~$M>0$.
	
	We conclude and re-iterate that both the introduction of the projection $\rhoMM$ of the consensus point and the employment of truncation of the noise variance $\left(\normm{V^i_t-v_{\alpha}(\empmeasure{t})}\wedge M\right)$ are main innovations to the original CBO method. We shall explain and justify these modifications in the following paragraph. 
 
	\revised{Despite these technical improvements, the approach to analyze the convergence behavior of the implementable scheme~\eqref{eq:17} follows a similar route already explored in \cite{carrillo2018analytical,carrillo2019consensus,fornasier2021consensus,fornasier2021convergence}.
    In particular, the convergence behavior of the method to the global minimizer~$v^*$ of the objective~$f$ is investigated on the level of the mean-field limit~\cite{huang2021MFLCBO,fornasier2021consensus} of the system~\eqref{eq:5}.
	More precisely, we study the macroscopic behavior of the agent density~$\rho\in\CC([0,T],\CP(\bbR^d))$, where $\rho_t=\Law(\overbar{V}_t)$ with
	\begin{equation}\label{eq:1}
		d\overbar{V}_t=-\lambda\left(\overbar{V}_t-\rhoM\right)dt+\sigma\left(\N{\overbar{V}_t-\vrho}_2\wedge M\right)d B_t
	\end{equation}
    and initial data $\overbar{V}_0\sim\rho_0$.
    Afterwards, by establishing a quantitative estimate on the mean-field approximation, i.e., the proximity of the mean-field system \eqref{eq:1} to the interacting particle system~\eqref{eq:5} and combining the two results, we obtain a convergence result for the CBO algorithm~\eqref{eq:17} with truncated noise.}

	\paragraph{Motivation for using truncated noise.}
	In what follows we provide a heuristic explanation of the theoretical benefits of employing a truncation in the noise of CBO as in~\eqref{eq:17}, \eqref{eq:5} and \eqref{eq:1}.
	Let us therefore first recall that the standard variant of CBO~\cite{pinnau2017consensus} can be retrieved from the model considered in this paper by setting $v_b=0$, $ R=\infty$ and $M=\infty$.
	For instance, in place of the mean-field dynamics \eqref{eq:1}, we would have 
	\begin{equation*}
		d \overbar{V}^\text{CBO}_t
		=-\lambda\left(\overbar{V}^\text{CBO}_t-\conspoint{\rho^\text{CBO}_t}\right)dt+\sigma\N{\overbar{V}^\text{CBO}_t-\conspoint{\rho^\text{CBO}_t}}_2 d B_t.
	\end{equation*}
	Attributed to the Laplace principle~\eqref{eq:laplace_principle} it holds $\conspoint{\rho^{\text{CBO}}_t}\approx\globmin$ for $\alpha$ sufficiently large, i.e., as $\alpha\rightarrow\infty$, the former dynamics converges to 
	\begin{equation}\label{eq:33}
		d \overbar{Y}^\text{CBO}_t=-\lambda\left(\overbar{Y}^\text{CBO}_t-\globmin\right)dt+\sigma\N{\overbar{Y}^\text{CBO}_t-\globmin}_2d B_t.
	\end{equation}
	Firstly, observe that here the first term imposes a direct drift to the global minimizer~$\globmin$ and thereby induces a contracting behavior, which is on the other hand counteracted by the diffusion term, which contributes a stochastic exploration around this point.
	In particular, with $\overbar{Y}^\text{CBO}_t$ approaching $\globmin$, the exploration vanishes so that $\overbar{Y}^\text{CBO}_t$ converges eventually deterministically to $\globmin$.
	Conversely, as long as $\overbar{Y}^\text{CBO}_t$ is far away from $\globmin$, the order of the random exploration is strong.
	%One question is when $Y_t$ is far away from $\globmin$, whether this order of  exploration is appropriate ? In the next, we will briefly analyze the above question.
	By It\^o's formula we have 
	\begin{equation*}
		\frac{d}{dt}\Ep{\N{\overbar{Y}^\text{CBO}_t-\globmin}_2^p}
		=p\left(-\lambda+\frac{\sigma^2}{2}\left(p+d-2\right)\right)\Ep{\N{\overbar{Y}^\text{CBO}_t-\globmin}_2^p}
	\end{equation*}
	and thus 
	\begin{equation}\label{eq:8}
		\Ep{\N{\overbar{Y}^\text{CBO}_t-\globmin}_2^p}
		= \exp\left(p\left(-\lambda+\frac{\sigma^2}{2}\left(p+d-2\right)\right)t\right) \Ep{\N{\overbar{Y}^\text{CBO}_0-\globmin}_2^p}
	\end{equation}
	for any $p\geq 1$.
	Denoting with $\mu^\text{CBO}_t$ the law of $\overbar{Y}^\text{CBO}_t$, this means that, given any $\lambda,\sigma>0$, there is some threshold exponent $p^*=p^*(\lambda,\sigma,d)$, such that
	\begin{equation*}
		\begin{split}
			\revised{\lim_{t\to\infty}} W_p\left(\mu^\text{CBO}_t,\delta_{\globmin}\right)
			&=\lim_{t\to\infty}\left(\Ep{\N{\overbar{Y}^\text{CBO}_t-\globmin}_2^p}\right)^{1/p} \\
			&=\lim_{t\to\infty}\exp\left(\left(-\lambda+\frac{\sigma^2}{2}\left(p+d-2\right)\right)t\right)\left(\Ep{\N{\overbar{Y}^\text{CBO}_0-\globmin}_2^p}\right)^{1/p}\\
			&=0
		\end{split}
	\end{equation*}
	for $p<p^*$,
	while for $p>p^*$ it holds
	\begin{equation*}
		\begin{split}
			\revised{\lim_{t\to\infty}} W_p\left(\mu^\text{CBO}_t,\delta_{\globmin}\right)
			&=\lim_{t\to\infty}\left(\Ep{\N{\overbar{Y}^\text{CBO}_t-\globmin}_2^p}\right)^{1/p}\\
			&=\lim_{t\to\infty}\exp\left(\left(-\lambda+\frac{\sigma^2}{2}\left(p+d-2\right)\right)t\right)\left(\Ep{\N{\overbar{Y}^\text{CBO}_0-\globmin}_2^p}\right)^{1/p}\\
			&=\infty.
		\end{split}
	\end{equation*}
    \revised{Recalling that the distribution of a random variable~$Y$ has heavy tails if and only if the moment generating function $M_Y(s):=\bbE\left[\exp(sY)\right] =\bbE\left[\sum_{p=0}^\infty (sY)^p/p!\right]$ is infinite for all $s>0$,}
	these computations suggest that the distribution of $\mu^\text{CBO}_t$ exhibits characteristics of heavy tails \revised{as $t\to\infty$,} thereby increasing the likelihood of encountering outliers in a sample drawn from $\mu^\text{CBO}_t$ \revised{for large $t$.}
	%Such disproportionate values have the potential to significantly distort statistical measures and such distortion should not be disregarded.
	
	On the contrary, for CBO with truncated noise~\eqref{eq:1}, we get, thanks once again to the Laplace principle as $\alpha\rightarrow\infty$, that \eqref{eq:1} converges to
	\begin{equation}\label{eq:1111}
		d\overbar{Y}_t
		=-\lambda\left(\overbar{Y}_t-\globmin\right)dt+\sigma\left(\N{\overbar{Y}_t-\globmin}_2\wedge M\right)dB_t,
	\end{equation}
	for which we can compute
	\begin{equation*}
		\begin{aligned}
			\frac{d}{dt}\Ep{\N{\overbar{Y}_t-\globmin}_2^p}
			&\leq -p\lambda\Ep{\N{\overbar{Y}_t-\globmin}_2^p}+p\frac{\sigma^2}{2}M^2\left(p+d-2\right)\Ep{\N{\overbar{Y}_t-\globmin}_2^{p-2}}\\
			&\leq -\lambda\Ep{\N{\overbar{Y}_t-\globmin}_2^p}+\lambda\frac{\sigma^pM^p(d+p-2)^{\frac{p}{2}}}{\lambda^{\frac{p}{2}}},
		\end{aligned}
	\end{equation*}
	for any $p\geq2$.
	Notice, that to obtain the second inequality we used Young's inequality\footnote{Choose $a=\lambda^{\frac{p-2}{p}} \mathbb{E}\left[\left\|\overbar{Y}_t-v^*\right\|^{p-2}_{\revised{2}}\right]$ and $b=\frac{\sigma^2 M^2(d+p-2)}{\lambda^{{(p-2)/p}}}$, and recall that $a b \leq \frac{p-2}{p} a^{\frac{p}{p-2}}+\frac{2}{p} b^{\frac{p}{2}}$.}
    as well as Jensen's inequality.
	By means of Gr\"onwall's inequality, we then have
	\begin{equation} \label{eq:12}
		\Ep{\N{\overbar{Y}_t-\globmin}_2^p}
		\leq \exp\left(-\lambda t\right) \Ep{\N{\overbar{Y}_0-\globmin}_2^p} + \frac{\sigma^p M^p(d+p-2)^{\frac{p}{2}}}{\lambda^{\frac{p}{2}}}
	\end{equation}
	and therefore, denoting with $\mu_t$ the law of $\overbar{Y}_t$,
	\begin{equation*}
		\lim_{t\to\infty}W_p\left(\mu_t,\delta_{\globmin}\right)\leq \frac{\sigma M\sqrt{d+p-2}}{\lambda^{\frac{1}{2}}}<\infty
	\end{equation*}
	for any $p\geq2$.
	
	In conclusion, we observe from Equation~\eqref{eq:8} that the standard CBO dynamics as described in \Cref{eq:33} diverges in the setting $\sigma^2d>2\lambda$ when considering the Wasserstein-$2$ distance~$W_2$.
	Contrarily, according to Equation~\eqref{eq:12}, the CBO dynamics with truncated noise as presented in Equation~\eqref{eq:1111} converges with exponential rate towards a neighborhood of $\globmin$, with radius $\sigma M\sqrt{d}/\sqrt{\lambda}$.
	This implies that for a relatively small value of $M$ the CBO dynamics with truncated noise exhibits greater robustness in relation to the parameter $\sigma^2d/\lambda$.
	This effect is confirmed numerically in Figure~\ref{fig:truncated_CBO_VS_CBO}.
	%These findings highlight the contrasting convergence properties of the two dynamics, shedding light on their behavior under different parameter regimes.
	
	\begin{remark}[Sub-Gaussianity of truncated CBO]
		An application of It\^o's formula
		allows to show that, for some $\revised{\kappa}>0$, $\Ep{\exp\left(\normmsq{\overbarscript{Y}_t-\globmin}/\revised{\kappa}^2\right)}<\infty$, provided $\Ep{\exp\left(\normmsq{\overbarscript{Y}_0-\globmin}/\revised{\kappa}^2\right)}<\infty$. 
		Thus, by incorporating a truncation in the noise term of the CBO dynamics, we ensure that the  resulting distribution $\mu_t$ exhibits sub-Gaussian behavior
		and therefore we enhance the regularity and well-behavedness of the statistics of $\mu_t$.
		As a consequence, more reliable and stable results when analyzing the properties and characteristics of the dynamics are to be expected.
	\end{remark}

	\paragraph{Contributions.}
	In view of the aforementioned enhanced regularity and well-behavedness of the statistics of CBO with truncated noise compared to standard CBO~\cite{pinnau2017consensus} together with the numerically observed improved performance as depicted in Figure~\ref{fig:truncated_CBO_VS_CBO},
	a rigorous convergence analysis of the implementable CBO algorithm with truncated noise as given in \eqref{eq:17} is of theoretical interest.
	In this work we provide  theoretical guarantees of global convergence of \eqref{eq:17} to the global minimizer~$\globmin$ for possibly nonconvex and nonsmooth objective functions~$\CE$.
    The approach to analyze the convergence behavior of the implementable scheme~\eqref{eq:17} follows a similar route as initiated and explored by the authors of \cite{carrillo2018analytical,carrillo2019consensus,fornasier2021consensus,fornasier2021convergence}.
    In particular, we first investigate the mean-field behavior \eqref{eq:1} \revised{of the system~\eqref{eq:5}.}
	Then, by establishing a quantitative estimate on the mean-field approximation, i.e., the proximity of the mean-field system \eqref{eq:1} to the interacting particle system~\eqref{eq:5}, we obtain a convergence result for the CBO algorithm~\eqref{eq:17} with truncated noise. Our proving technique nevertheless differs in crucial parts from the one in \cite{fornasier2021consensus,fornasier2021convergence} as, on the one side, we do take advantage of the truncations, and, on the other side, we require additional technical effort to exploit and deal with the enhanced flexibility of the truncated model. Specifically, the central novelty can be identified in the proof of sub-Gaussianity of the process, see Lemma \ref{thm:2}.
	\begin{figure}[H]
		\centering
		\begin{subfigure}[t]{0.47\textwidth}
			\centering
			\includegraphics[trim=54 218 65 236,clip,height=0.2\textheight]{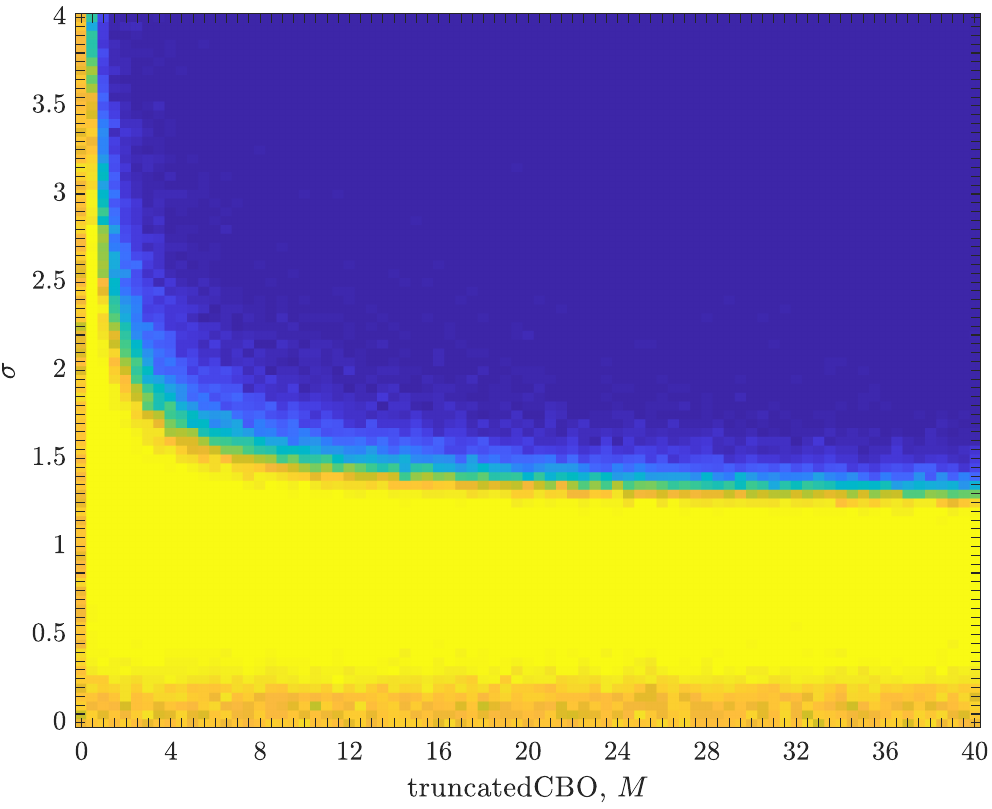}%
			\vspace{0.06cm}
			\includegraphics[trim=260 218 292 236,clip,height=0.2\textheight]{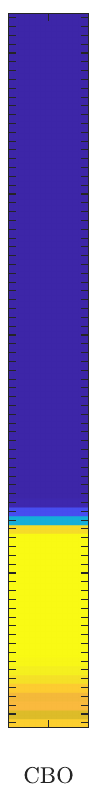}
			\vspace{0.06cm}
			\includegraphics[trim=488 218 65 236,clip,height=0.2\textheight]{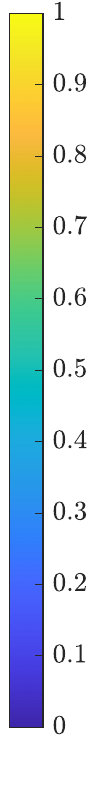}
			\caption{Phase diagram of success probabilities of isotropic CBO with and without truncated noise at the example of the Ackley function $\CE(v) = -20\exp\big(\!-\!{0.2}/{\sqrt{d}}\N{v}_2\big) - \exp\left({1}/{d}\sum_{k=1}^d \cos(2\pi v_k)\right)$ with $d=4$}
			\label{fig:truncated_CBO_VS_CBO_Ackley}
		\end{subfigure}~\hspace{1em}~
		\begin{subfigure}[t]{0.47\textwidth}
			\centering
			\includegraphics[trim=54 218 65 236,clip,height=0.2\textheight]{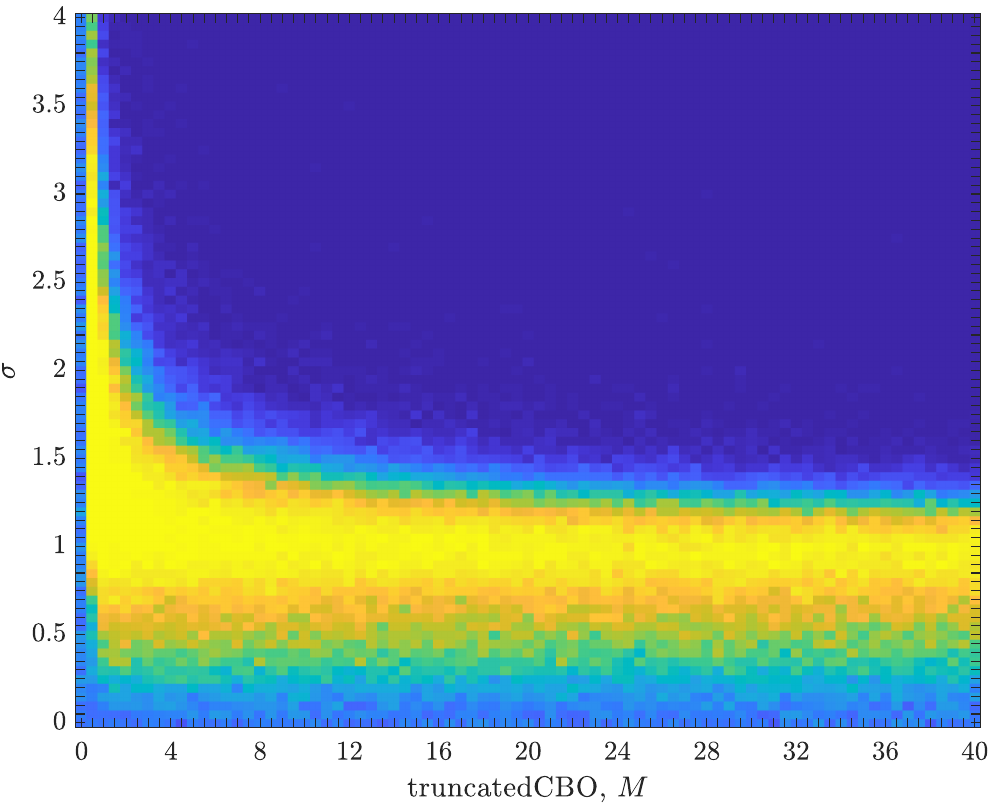}%
			\vspace{0.06cm}
			\includegraphics[trim=260 218 292 236,clip,height=0.2\textheight]{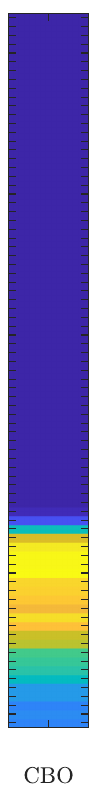}
			\vspace{0.06cm}
			\includegraphics[trim=488 218 65 236,clip,height=0.2\textheight]{colorbar.pdf}
			\caption{Phase diagram of success probabilities of isotropic CBO with and without truncated noise at the example of the Rastrigin function $\CE(v) = \sum_{k=1}^d v_k^2 + 2.5 \big(1-\cos(2\pi v_k)\big)$ with $d=4$}
			\label{fig:truncated_CBO_VS_CBO_Rastrigin}
		\end{subfigure}
		\caption{\footnotesize A comparison of the success probabilities of isotropic CBO with (left phase diagrams) and without (right separate columns) truncated noise for different values of the truncation parameter~$M$ and the noise level~$\sigma$.
			(Note that standard CBO as investigated in \cite{pinnau2017consensus,carrillo2018analytical,fornasier2021consensus} is retrieved when choosing $M=\infty$, $R=\infty$ and $v_b=0$ in \eqref{eq:17}).
			In both settings \textbf{(a)} and \textbf{(b)} the depicted success probabilities are averaged over $100$ runs and the implemented scheme is given by an Euler-Maruyama discretization of Equation~\eqref{eq:5} with time horizon~$T=50$, discrete time step size~$\Delta t=0.01$, $R=\infty$, $v_b=0$, $\alpha=10^5$ and $\lambda=1$.
			We use $N=100$ particles, which are initialized according to $\rho_0=\CN((1,\dots,1),2000)$.
			In both figures we plot the success probability of standard CBO (right separate column) and the CBO variant with truncated noise (left phase transition diagram) for different values of the truncation parameter~$M$ and the noise level $\sigma$, when optimizing the Ackley (\textbf{(a)}) and Rastrigin (\textbf{(b)}) function, respectively.
			\revised{We observe that truncating the noise term (by decreasing $M$) consistently allows for a wider flexibility when choosing the noise level $\sigma$ and thus increasing the likelihood of successfully locating the global minimizer.}}
		\label{fig:truncated_CBO_VS_CBO}
	\end{figure}
	
	\subsection{Organization} \label{subsec:organization}
	In Section~\ref{sec:main_result} we present and discuss our main theoretical contribution about the global convergence of CBO with truncated noise in probability and expectation.
	Section~\ref{sec:proofs} collects the necessary proof details for this result.
	In Section~\ref{sec:numerics} we numerically demonstrate the benefits of using truncated noise, before we provide a conclusion of the paper in Section~\ref{sec:conclusions}.
	For the sake of reproducible research, in the GitHub repository \url{https://github.com/KonstantinRiedl/CBOGlobalConvergenceAnalysis} we provide the Matlab code implementing CBO with truncated noise.

	\subsection{Notation} \label{subsec:notation}
	We use $\N{\,\cdot\,}_2$ to denote the Euclidean norm on $\mathbb{R}^d$.
	Euclidean balls are denoted as $B_{r}(u) \!:=\! \{v \in \bbR^d: \Nnormal{v-u}_2 \leq r\}$.
	For the space of continuous functions~$f:X\rightarrow Y$ we write $\CC(X,Y)$, with $X\subset\bbR^n$ and a suitable topological space $Y$.
	For an open set $X\subset\bbR^n$ and for $Y=\bbR^m$ the spaces~$\CC^k_{c}(X,Y)$ and~$\CC^k_{b}(X,Y)$ contain functions~$f\in\CC(X,Y)$ that are $k$-times continuously differentiable and have compact support or are bounded, respectively.
	We omit $Y$ in the real-valued case.
	All stochastic processes are considered on the probability space $\left(\Omega,\mathscr{F},\mathbb{P}\right)$.
	The main objects of study are laws of such processes, $\rho\in\CC([0,T],\CP(\bbR^d))$, where the set $\CP(\bbR^d)$ contains all Borel probability measures over $\bbR^d$.
	With $\rho_t\in\CP(\bbR^d)$ we refer to a snapshot of such law at time~$t$.
	Measures~$\indivmeasure \in \CP(\bbR^d)$ with finite $p$-th moment $\int \Nnormal{v}_2^p\,d\indivmeasure(v)$ are collected in $\CP_p(\bbR^d)$.
	For any $1\leq p<\infty$, $W_p$ denotes the \mbox{Wasserstein-$p$} distance between two Borel probability measures~$\indivmeasure_1,\indivmeasure_2\in\CP_p(\bbR^d)$, see, e.g., \cite{savare2008gradientflows}. $\Ep{\cdot}$ denotes the expectation.

	\section{Global Convergence of CBO with Truncated Noise} \label{sec:main_result}
	
	We now present the main theoretical result of this work about the global convergence of CBO with truncated noise for objective functions that satisfy the following conditions.
	
	%	\MF{Condition A3 is pretty strong, as it does not allow for super-quadratically growing functions, and this is a bit strange as one can expect that, with faster growth of $f$, CBO would be able to work better/easier. I see that A3 is used crucially in formula (64). I wonder whether there is a fix for faster growing functions. All depends on sub-Gaussianity of the law. However, I do not see how can we get better concentration as the drift term will always give in (44) only quadratic terms in (45).}
	\begin{definition}[Assumptions] \label{def:assumptions}
		Throughout we are interested in functions $\CE \in \CC(\bbR^d)$, for which
		\begin{enumerate}[label=A\arabic*,labelsep=10pt,leftmargin=35pt]
			\item\label{asp:11} there exist $\globmin\in\bbR^d$ such that $\CE(\globmin)=\inf_{v\in\bbR^d} \CE(v)=:\underbar{\CE}$ and $\underbar{\alpha},L_u>0$ such that
			\begin{align}\label{eq:LU}
				\sup_{v\in\mathbb{R}^d}\normm{ve^{-\alpha (f(v)-\underline{f})}}=:L_u<\infty
			\end{align}
			for any $\alpha\geq\underbar{\alpha}$ and any $v\in\mathbb{R}^d$,
			\item\label{asp:22} there exist $\CE_{\infty},R_0,\nu,L_\nu > 0$ such that
			\begin{align}
				\label{eq:asm_icp_vstar}
				\N{v-\globmin}_2 &\leq \frac{1}{L_\nu}(\CE(v)-\minobj)^{\nu} \quad \text{ for all } v \in B_{R_0}(\globmin),\\
				\CE_{\infty} &< \CE(v)-\minobj\quad \text{ for all } v \in \big(B_{R_0}(\globmin)\big)^c,
			\end{align}
			\item\label{asp:33} there exist $L_{\gamma}>0,\gamma\in [0,1]$ such that
			\begin{align}
				\abs{f(v)-f(w)} &\leq L_{\gamma}(\normm{v-\globmin}^{\gamma}+\normm{w-\globmin}^{\gamma})\normm{v-w} \quad \text{ for all } v, w \in \mathbb{R}^d,\\
				f(v)-\underline{f} &\leq L_{\gamma}\left(1+\normm{v-\globmin}^{1+\gamma}\right) \quad \text{ for all } v \in \mathbb{R}^d.
			\end{align}
		\end{enumerate}
	\end{definition}
    
    \noindent
    A few comments are in order:
    Condition \ref{asp:11} establishes the existence of a minimizer $\globmin$ and requires a certain growth of the function $f$.
    Condition \ref{asp:22} ensures that the value of the function $f$ at a point $v$ can locally be an indicator of the distance between $v$ and the minimizer $\globmin$.
    This error bound condition was first introduced in \cite{fornasier2021consensus} under the name inverse continuity condition.
    \revised{It in particular guarantees the uniqueness of the global minimizer~$\globmin$.}
    Condition \ref{asp:33} sets controllable bounds on the local Lipschitz constant of $f$ and on the growth of $f$, which is required to be at most quadratic. A similar requirement appears also in \cite{carrillo2018analytical,fornasier2021consensus}, but there also a quadratic lower bound was imposed.

 %\Cref{eq:LU} in condition \ref{asp:11} is very weak and new here. Condition \ref{asp:22} was first shown in \cite{fornasier2021consensus} and used to guarantee a quantitative Laplace principle. Condition \ref{asp:33} with $\gamma=1$ was used in \cite{fornasier2021consensus,carrillo2018analytical} to guranttee the existense and uniqueness of the mean-field limit CBO dynamics, additionally, apart from the upper bound, they imposed an additional requirement on the function $f$, namely, quadratic growth outside a ball.

 %{\color{blue}Add comments comparing with previous work ... }

	\subsection{Main Result}
	
	We can now state the main result of the paper.
	Its proof is deferred to Section~\ref{sec:proofs}.
	
	\begin{theorem}\label{thm:main1}
		Let $\CE \in \CC(\bbR^d)$ satisfy \ref{asp:11}, \ref{asp:22} and \ref{asp:33}.
		Moreover, let $\rho_0\in\CP_4(\bbR^d)$ with $\globmin\in\supp(\rho_0)$.
		Let $V^i_{\revised{0,\Delta t}}$ be sampled i.i.d.\@ from $\rho_0$ and denote by $((V^i_{\revised{{k,\Delta t}}})_{k=1,\dots,K})_{i=1,\dots,N}$ the iterations generated by the numerical scheme~\eqref{eq:17}.
		Fix any $\epsilon\in(0,W_2^2\left(\rho_0,\delta_{\globmin}\right))$, define the time horizon
		\begin{align*}
			T^*:=\frac{1}{\lambda}\log\left(\frac{2W_2^2\left(\rho_0,\delta_{\globmin}\right)}{\epsilon}\right)
		\end{align*}
		and let $K \in \mathbb{N}$ and $\Delta t$ satisfy ${{K\Delta t}}=T^*$.
		Moreover, let $R\in \big(\!\normm{v_b-\globmin}+\sqrt{\epsilon/2},\infty\big)$, $M\in (0,\infty)$ and $\lambda,\sigma>0$ be such that $\lambda\geq 2\sigma^2d$ or $\sigma^2M^2d=\mathcal{O}(\epsilon)$.
		Then, by choosing $\alpha$ sufficiently large and $N\geq(16\alpha L_{\gamma}\sigma^2M^2)/\lambda$, it holds 
		% there is $K_0\geq (2\alpha L_{\gamma})^{-\frac{1}{2}}$ such that $C_{K_0}:=\sup_{t\in[0,T^*]}\Ep{e^{\frac{\normmsq{\bV_t-\globmin}}{K_0^2}}}<\infty$, we also have
 \begin{equation}\label{eq:7}
			\Ep{\normmsq{\frac{1}{N} \sum_{i=1}^N V^i_{\revised{K, \Delta t}}-\globmin}}
			\revised{\lesssim} C_{\mathrm{NA}}(\Delta t)^{2m}+\frac{C_{\mathrm{MFA}}}{N}+\epsilon
		\end{equation}
        \revised{up to a generic constant.}
        Here, $C_{\mathrm{NA}}$ depends linearly on the dimension $d$ and the number of particles $N$ and exponentially on the time horizon $T^*$,
		$m$ is the order of accuracy of the numerical scheme (for the Euler-Maruyama scheme $m = 1/2$),
		and $C_{\mathrm{MFA}} = C_{\mathrm{MFA}}(\lambda,\sigma,d,\alpha,L_{\nu},\nu,L_{\gamma},L_u,T^*,R,v_b,v^*,M)$.
	\end{theorem}
	
	\begin{remark}
		In the statement of \Cref{thm:main1}, the parameters $R$ and $v_b$ play a crucial role.
        We already mentioned how they can be chosen in an example after Equation~\eqref{defproj}.
        The role of these parameters is bolstered in particular in the proof of \Cref{thm:main1}, where it is demonstrated that, by selecting a sufficiently large $\alpha$ depending on $R$ and $v_b$, the dynamics~\eqref{eq:1} can be set equal to  %\MF{Here there was typo I guess: in fact the projection here $\mathcal{P}_{\globmin,\delta}(\rho_t)$ was applied to $\rho_t$ and not to a vector. I corrected it, but please check it. Same for the formula below.}
		\begin{equation*}
			d \overbar{V}_t=-\lambda\left(\overbar{V}_t-\mathcal{P}_{\globmin,\delta}(v_\alpha(\rho_t))\right)dt+\sigma\left(\left\|\overbar{V}_t-\vrho\right\|_2\wedge M\right)d B_t,,
		\end{equation*}
		where $\delta$ represents a small value.
        For the dynamics~\eqref{eq:5}, we can analogously establish its equivalence to
		\begin{equation*}
			d V^i_t=-\lambda\left(V^i_t-\mathcal{P}_{\globmin,\delta}(v_\alpha(\widehat{\rho}_t^N))\right)dt+\sigma\left(\normm{V^i_t-v_{\alpha}(\widehat{\rho}^N_t)}\wedge M\right)dB^i_t,\quad \text{$i=1,\dots,N$},
		\end{equation*}
		with high probability, contingent upon the selection of sufficiently large values for both $\alpha$ and $N$.
	\end{remark}
	
	\begin{remark}
		The convergence result in form of Theorem~\ref{thm:main1}  obtained in this work differs from the one presented in \cite[Theorem 14]{fornasier2021consensus}
		in the sense that we obtain convergence is in expectation, while in \cite{fornasier2021consensus} convergence with high probability is established.
		This distinction arises from the truncation of the noise term employed in our algorithm.
	\end{remark}

	\section{Proof Details for Section~\ref{sec:main_result}} \label{sec:proofs}
	
	\subsection{Well-Posedness of Equations~\eqref{eq:17} and \eqref{eq:5}}
	
	With the projection map~$\CP_{v_b,R}$ being $1$-Lipschitz, existence and uniqueness of strong solutions to the SDEs~\eqref{eq:17} and \eqref{eq:5} are assured by essentially analogous proofs as in \cite[Theorems 2.1,  3.1 and 3.2]{carrillo2018analytical}.
	The details shall be omitted.
	Let us remark, however, that due to the presence of the truncation and the projection map, we do not require the function $\CE$ to be bounded from above or exhibit quadratic growth outside a ball, as required in \cite[Theorems 2.1, 3.1 and 3.2]{carrillo2018analytical}.
	%	\MF{Right, but assumption A3 does not even allow for quadratic growth! That condition is really restrictive.} 
	
	\subsection{Proof Details for Theorem~\ref{thm:main1}}
	
	\begin{remark} \label{remark:wlog}
		Since adding some constant offset to $\CE$ does not affect the dynamics of \Cref{eq:1,eq:5}, we will assume $\underbar{f}=0$ in the proofs for simplicity but without loss of generality.
	\end{remark}

	%Recall that $V^i_{\revised{{k,\Delta t}}}$ denote the numerical approximation of the SDE system \eqref{eq:5}, where $\Delta t$ is the step-size and $k$ denotes the $k$th step. Then we have the following convergence guarantee.
	%	
	\noindent
    Let us first provide a sketch of the proof of \Cref{thm:main1}.
	For the approximation error  \eqref{eq:7} we have the error decomposition
	\begin{equation} \label{eq:error_decomposition}
		\begin{aligned}
			\Ep{\normmsq{\frac{1}{N} \sum_{i=1}^N V^i_{\revised{K, \Delta t}}-\globmin}}
            &\lesssim \underbrace{\Ep{\normmsq{\frac{1}{N} \sum_{i=1}^N \left(V^i_{\revised{K, \Delta t}}-V_{T^*}^i\right)}}}_{I}+\underbrace{\Ep{\normmsq{\frac{1}{N} \sum_{i=1}^N \left(V_{T^*}^i-\bV_{T^*}^i\right)}}}_{II}\\
			&\qquad\qquad\quad+\underbrace{\Ep{\normmsq{\frac{1}{N} \sum_{i=1}^N \bV_{T^*}^i-\globmin}}}_{III},
		\end{aligned}
	\end{equation}
    \revised{where $((\bV_t^i)_{t\geq0})_{i=1,\dots,N}$ denote $N$ independent copies of the mean-field process $(\bV_t)_{t\geq0}$ satisfying \Cref{eq:1}.}
    
	In what follows, we investigate each of the three term separately.
    Term $I$ can be bounded by $C_{\mathrm{NA}}\left(\Delta t\right)^{2m}$ using classical results on the convergence of numerical schemes for stochastic differential equations (SDEs), as mentioned for instance in \cite{platen1999introduction}.
    The second and third term, respectively, are analyzed in separate subsections, providing detailed explanations and bounds for each of the two terms $II$ and $III$.
	
    {\itshape Before doing so, let us provide a concise guide for reading the proofs.
    As the proofs are quite technical, we start for reader's convenience by presenting the main building blocks of the result first, and collect the more technical steps in subsequent lemmas.
    This arrangement should hopefully allow to grasp the structure of the proof more easily, and to dig deeper into the details along with the reading.}

	\subsubsection{Upper Bound for the Second Term in \eqref{eq:error_decomposition}}
 
	For Term $II$ of the error decomposition~\eqref{eq:error_decomposition} we have the following upper bound.
	\begin{proposition}
		Let $\CE \in \CC(\bbR^d)$ satisfy \ref{asp:11}, \ref{asp:22} and \ref{asp:33}.
        Moreover, let $R$ and $M$ be finite such that $R\geq \normm{v_b-v^*}$ and let $N\geq (16\alpha L_{\gamma}\sigma^2M^2)/\lambda$.
        Then we have
		\begin{equation}
			\Ep{\normmsq{\frac{1}{N} \sum_{i=1}^N \left(V_{T^*}^i-\bV_{T^*}^i\right)}}\leq \frac{C_{\mathrm{MFA}}}{N},
		\end{equation}
		where $C_{\mathrm{MFA}} = C_{\mathrm{MFA}}(\lambda,\sigma,d,\alpha,L_{\nu},\nu,L_{\gamma},L_u,T^*,R,v_b,v^*,M)$.
	\end{proposition}
	\begin{proof}
		By a synchronous coupling we have
		\begin{align*}
			d \overbar{V}^i_t &=-\lambda\left(\overbar{V}^i_t-\rhoM\right)dt+\sigma\left(\normm{\overbar{V}^i_t-\vrho}\wedge M\right)d B^i_t,\\
			d {V}^i_t &=-\lambda\left({V}^i_t-\rhoMM\right)dt+\sigma\left(\normm{\overbar{V}^i_t-v_{\alpha}(\empmeasure{t})}\wedge M\right)d B^i_t,
		\end{align*}
		with coinciding Brownian motions.
        Moreover, recall that  $\Law(\overbar{V}_t^i)=\rho_t$ and $\empmeasure{t}={1}/{N}\sum_{i=1}^N\delta_{V_t^i}$.
		By It\^o's formula we then have
		\begin{equation}\label{eq:30}
			\begin{aligned}
				d\normmsq{\bV^i_t-V_t^i}&
				%=2\left(\bV^i_t-V_t^i\right)^{\top}d\left(\bV^i_t-V_t^i\right)+\sigma^2d\left(\normm{\overbar{V}^i_t-\vrho}\wedge M-\normm{V_t^i-{v}_{\alpha}(\empmeasure{t})}\wedge M\right)^2dt\\
				=\Big(-2\lambda\inner{\bV^i_t-V_t^i}{\left(\bV^i_t-V_t^i\right)-\left(\rhoM-\rhoMM\right)}\\
				&\quad\,+\sigma^2d\left(\normm{\overbar{V}^i_t-\vrho}\wedge M-\normm{V_t^i-{v}_{\alpha}(\empmeasure{t})}\wedge M\right)^2\Big)\,dt\\
				&\quad\,+2\sigma\left(\normm{\overbar{V}^i_t-\vrho}\wedge M-\normm{V_t^i-{v}_{\alpha}(\empmeasure{t})}\wedge M\right)\left(\bV^i_t-V_t^i\right)^{\top}dB^i_t,
			\end{aligned}
		\end{equation}
		and after taking the expectation on both sides
		\begin{equation}\label{eq:31}
			\begin{aligned}
				\frac{d}{dt}\Ep{\normmsq{\bV^i_t-V_t^i}}&=-2\lambda\Ep{\inner{\bV^i_t-V_t^i}{\left(\bV^i_t-V_t^i\right)-\left(\rhoM-\rhoMM\right)}}\\
				&\quad\,+\sigma^2d\Ep{\left(\normm{\overbar{V}^i_t-\vrho}\wedge M-\normm{V_t^i-{v}_{\alpha}(\empmeasure{t})}\wedge M\right)^2}\\
				&\leq -2\lambda\Ep{\normmsq{\bV^i_t-V_t^i}}+\sigma^2d\Ep{\normmsq{\left(\bV^i_t-V_t^i\right)-\left(\vrho-{v}_{\alpha}(\empmeasure{t})\right)}}\\
				&\quad\,+2\lambda\Ep{\normm{\bV^i_t-V_t^i}\normm{\rhoM-\rhoMM}}\\
				&\leq -2\lambda\Ep{\normmsq{\bV^i_t-V_t^i}}+2\lambda\Ep{\normm{\bV^i_t-V_t^i}\normm{\vrho-{v}_{\alpha}(\empmeasure{t})}}\\
				&\quad\,+\sigma^2d\Ep{\normmsq{\left(\bV^i_t-V_t^i\right)-\left(\vrho-{v}_{\alpha}(\empmeasure{t})\right)}}.
			\end{aligned}
		\end{equation}
        \revised{Here, let us remark that the last (stochastic) term in \eqref{eq:30} disappears after taking the expectation. This is due to $\Ep{\normmsq{\bV_t^i-V_t^i}}<\infty$, which can be derived from Lemma~\ref{thm:2} after noticing that Lemma~\ref{thm:2} also holds for processes $V_t^i$.}
		Since by Young's inequality it holds
		\begin{equation*}
			2\lambda\Ep{\normm{\bV^i_t-V_t^i}\normm{\vrho-{v}_{\alpha}(\empmeasure{t})}}\leq \lambda\left(\frac{\Ep{\normmsq{\bV^i_t-V_t^i}}}{2}+2\Ep{\normmsq{\vrho-{v}_{\alpha}(\empmeasure{t})}}\right),
		\end{equation*}
		and
		\begin{equation*}
			\Ep{\normmsq{\left(\bV^i_t-V_t^i\right)-\left(\vrho-{v}_{\alpha}(\empmeasure{t})\right)}}\leq 2 \Ep{\normmsq{\bV^i_t-V_t^i}+\normmsq{\vrho-{v}_{\alpha}(\empmeasure{t})}},
		\end{equation*}
		we obtain
		\begin{equation}
			\begin{aligned}
				\frac{d}{dt}\Ep{\normmsq{\bV^i_t-V_t^i}}&\leq \left(-\frac{3\lambda}{2}+2\sigma^2d\right)\Ep{\normmsq{\bV^i_t-V_t^i}}\\
				&\quad\,+2\left(\lambda+\sigma^2d\right)\Ep{\normmsq{\vrho-{v}_{\alpha}(\empmeasure{t})}}
			\end{aligned}
		\end{equation}
        after inserting the former two inequalities into \Cref{eq:31}.
		For the term $\Ep{\normmsq{\vrho-{v}_{\alpha}(\empmeasure{t})}}$ we can decompose
		\begin{equation}\label{eq:2727}
			\begin{aligned}
				\Ep{\normmsq{\vrho-{v}_{\alpha}(\empmeasure{t})}}\leq 2\Ep{\normmsq{\vrho-{v}_{\alpha}(\bar{\rho}^N_t)}}+2\Ep{\normmsq{{v}_{\alpha}(\bar{\rho}^N_t)-{v}_{\alpha}(\empmeasure{t})}},
			\end{aligned}
		\end{equation}
		where we denote
		\begin{equation*}
			%{v}_{\alpha}(\bar{\rho}^N_t):=\frac{\frac{1}{N}\sum \limits_{i=1}^N\bV_t^ie^{-\alpha f(\bV_t^i)}}{\frac{1}{N}\sum\limits_{i=1}^Ne^{-\alpha f(\bV_t^i)}};
            \bar{\rho}^N_t=\frac{1}{N}\sum_{i=1}^N\delta_{\overbarscript{V}_t^i}.
		\end{equation*}
		For the first term in \Cref{eq:2727}, by Lemma \ref{lem:332}, we have 
		\begin{equation*}
			\Ep{\normmsq{\vrho-{v}_{\alpha}(\bar{\rho}^N_t)}}\leq {C_0}\frac{1}{N}
		\end{equation*}
		for some constant $C_0$ depending on $\lambda,\sigma,d,\alpha,L_{\gamma},L_u,T^*,R,v_b,v^*$ and $M$.
        For the second term in \Cref{eq:2727}, by combining \cite[Lemma 3.2]{carrillo2018analytical} and Lemma \ref{thm:2}, we obtain
		\begin{equation*}
			\Ep{\normmsq{{v}_{\alpha}(\bar{\rho}^N_t)-{v}_{\alpha}(\empmeasure{t})}}\leq C_1\frac{1}{N}\sum_{i=1}^N\Ep{\normmsq{\bV^i_t-V_t^i}},
		\end{equation*}
		for some constant $C_1$ depending on $\lambda,\sigma,d,\alpha,L_u,R$ and $M$.
        %{(As previously mentioned, here we preferred to postpone  technical steps and to collect them in Lemma \ref{lem:332} in order to allow the reader to graps the structure of the proof first. This presentation approach will continue with this style along the paper.)}
		Combining these estimates we conclude
		\begin{equation*}
			\begin{aligned}
				\frac{d}{dt}\frac{1}{N}\sum_{i=1}^N\Ep{\normmsq{\bV^i_t-V_t^i}}&\leq \left(-\frac{3\lambda}{2}+2\sigma^2d+4C_1\left(\lambda+\sigma^2 d\right)\right)\frac{1}{N}\sum_{i=1}^N\Ep{\normmsq{\bV^i_t-V_t^i}}\\
				&\quad\,+ 4\left(\lambda+\sigma^2 d\right)C_0\frac{1}{N}.
			\end{aligned}
		\end{equation*}
	 After an application of Gr\"onwall's inequality \revised{and noting that $\bV^i_0=V^i_0$ for all $i=1,\dots,N$, we have}
		\begin{equation}
			\begin{aligned}
				\frac{1}{N}\sum_{i=1}^N\Ep{\normmsq{\bV^i_t-V_t^i}}&\leq 4\left(\lambda+\sigma^2 d\right)\frac{C_0}{N}te^{\left(-\frac{3\lambda}{2}+2\sigma^2d+4C_1(\lambda+\sigma^2 d)\right)t}.
			\end{aligned}
		\end{equation}
        for any $t\in [0,T^*]$.
		Finally, by Jensen's inequality and letting $t=T^*$, we have 
		\begin{equation}
			\Ep{\normmsq{\frac{1}{N} \sum_{i=1}^N \left(V_{T^*}^i-\bV_{T^*}^i\right)}}\leq \frac{C_{\mathrm{MFA}}}{N},
		\end{equation}
		where the constant $C_{\mathrm{MFA}}$ depends on $\lambda,\sigma,d,\alpha,L_u,L_{\gamma},T^*,R,v_b,v^*$ and $M$. 
	\end{proof}

    \noindent
	In the next lemma we show that the distribution of $\overbar{V}_t$ is sub-Gaussian.
	\begin{lemma}\label{thm:2}
		Let $R$ and $M$ be finite with $R\geq \normm{v_b-v^*}$. For any $\revised{\kappa}>0$, let $N$ satisfy $N\geq{(4 \sigma^2 M^2)}/{(\lambda \revised{\kappa}^2)}$.
        Then, provided that $\Ep{\exp({\sum_{i=1}^N\normm{\overbar{V}^i_0-\globmin}^2}/{(N\revised{\kappa}^2)})}<\infty$, it holds
		\begin{equation}
			\gC:=	\sup_{t\in [0,T^*]}\Ep{\exp\left(\frac{\sum_{i=1}^N\normm{\overbarscript{V}^i_t-\globmin}^2}{N\revised{\kappa}^2}\right)} <\infty,
		\end{equation}
		where $\gC$ depends on $\revised{\kappa},\lambda,\sigma,d,R,M$ and $T^*$, and where
		\begin{equation*}
			d \overbar{V}^i_t =-\lambda\left(\overbar{V}^i_t-\rhoM\right)dt+\sigma\left(\normm{\overbar{V}^i_t-\vrho}\wedge M\right)d B^i_t
		\end{equation*}
        for $i=1,\dots,N$ with
		$B^i_t$ being independent to each other and $\Law(\overbar{V}_t^i)=\rho_t$.
	\end{lemma} 
	
	\begin{proof}
        To apply It\^o's formula, we need to truncate the function $\exp({\normmsq{v}}/{\revised{\kappa}^2})$ from above.
		\revised{For this,} define \revised{for $W>0$ the function}
		\begin{equation*}
			\ff(x):=\begin{cases}
				x& x\in[0,W-1]\\
				\frac{1}{16}(x+1-W)^4-\frac{1}{4}(x+1-W)^3+x&x\in[W-1,W+1]\\
				W&x\in [W+1,\infty)
			\end{cases}.
		\end{equation*}
		It is easy to verify that $\ff$ is a $\mathcal{C}^2$ approximation of the function $x\wedge W$ satisfying $ \ff\in \mathcal{C}^2(\mathbb{R}^+)$, $\ff(x)\leq x\wedge W$, $ 
		\ff^{\prime}\in [0,1]$ and $\ff^{\prime\prime}\leq 0$.
		%		\begin{figure}
			%			\centering
			%			\includegraphics[width=0.5\linewidth]{compare}
			%			\caption{The comparasion between function $x\wedge 5$ and $f_5$.}
			%			\label{fig:compare}
			%		\end{figure}
		
		Since $\Gwn:=\exp({\ff\big(\!\sum_{i=1}^N\normmsq{\bV^i_t-\globmin}/N\big)}/{\revised{\kappa}^2})$ is upper bounded, we can apply It\^o's formula to it.
        We abbreviate $\ff^{\prime}:=\ff^{\prime}\big(\!\sum_{i=1}^N\normmsq{\bV_t^i\revised{-\globmin}}/N\big)$ and $\ff^{\prime\prime}:=\ff^{\prime\prime}\big(\!\sum_{i=1}^N\normmsq{\overbar{V}_t^i}/N\big)$ in what follows.
        With the notation $Y_t:=\left((\overbar{V}_t^1)^{\top},\cdots,(\overbar{V}_t^N)^{\top}\right)^{\top}$, the $Nd$ dimensional process $Y_t$ satisfies $ dY_t=-\lambda\big(Y_t-\overline{\mathcal{P}_{v_b,R}(\rho_t)}\big)\,dt+\mathcal{M}dB_t$, where $\overline{\mathcal{P}_{v_b,R}(\rho_t)}=\left({\mathcal{P}_{v_b,R}(\rho_t)}^{\top},\ldots,{\mathcal{P}_{v_b,R}(\rho_t)}^{\top}\right)^{\top}$, $\mathcal{M}=\operatorname{diag}\left(\mathcal{M}_1,\ldots,\mathcal{M}_N\right)$ with $\mathcal{M}_i=\sigma\normm{\overbar{V}^i_t-v_\alpha\left(\rho_t\right)}\wedge M \mathrm{I}_{d}$ and $B_t$ the $Nd$ dimensional Brownian motion.
        We then have $\Gwn=\exp\left(G_W\big(\!\normmsq{Y_t}/N\big)/\revised{\kappa}^2\right)$ and 
		%\MF{Here I would suggest to add, for reader convenience, that the consensus point is projected onto a ball of radius $R$ that contains the minimizer and this allow estimating all with respect to $\globmin$ as in (45).}
	\begin{equation}\label{eq:crux1}
			\begin{aligned}
				d \Gwn&= \sum_{i=1}^N\nabla_{Y_t}\Gwn dY_t+\frac{1}{2}\operatorname{tr}\left(\mathcal{M}\nabla_{Y_t,Y_t}^2 \Gwn \mathcal{M}\right)dt\\
                &=\Gwn\frac{\ff^{\prime}}{\revised{\kappa}^2}\sum_{i=1}^N\left(2\frac{\overbar{V}^i_t-\globmin}{N}\right)^{\top}d\overbar{V}^i_t\\
				&\quad\,+\frac{1}{2}\Gwn\sum_{i=1}^N\left(\ff^{\prime}\frac{2d}{N\revised{\kappa}^2}+\ff^{\prime\prime}\frac{4\normm{\overbar{V}^i_t-\globmin}^2}{N^2\revised{\kappa}^2}\right.\\
				&\quad\,\left.+\left(\ff^{\prime}\right)^2\frac{4\normm{\overbar{V}^i_t-\globmin}^2}{N^2\revised{\kappa}^4}\right)\left(\sigma\normm{\overbar{V}^i_t-v_\alpha\left(\rho_t\right)}\wedge M\right)^2dt.\\
			\end{aligned}
		\end{equation}
  The first term on the right-hand side of \eqref{eq:crux1} can be expanded as follows
\begin{equation}\label{eq:477}
    \begin{aligned}
        &\Gwn\frac{\ff^{\prime}}{\revised{\kappa}^2}\sum_{i=1}^N\left(2\frac{\overbar{V}^i_t-\globmin}{N}\right)^{\top}d\overbar{V}^i_t
        =\Gwn\ff^{\prime}\sum_{i=1}^N\left(2\frac{\overbar{V}^i_t-\globmin}{N\revised{\kappa}^2}\right)^{\top}d\overbar{V}^i_t\\
        &\quad=\Gwn\ff^{\prime}\sum_{i=1}^N\left(2\frac{\overbar{V}^i_t-\globmin}{N\revised{\kappa}^2}\right)^{\top}\left(-\lambda\left(\bar{V_t}^i-v^*+{v^*-\mathcal{P}_{v_b,R}(\rho_t)})\right)dt+\sigma\left(\normm{\overbar{V}^i_t-\vrho}\wedge M\right)d B^i_t\right)\\
        &\quad= \Gwn\ff^{\prime}\left\{\frac{-2\lambda}{N\revised{\kappa}^2}\sum_{i=1}^N\normmsq{\overbar{V}_t^i-v^*}dt-\frac{2\lambda}{N\revised{\kappa}^2}\sum_{i=1}^N\inner{\overbar{V}_t^i-v^*}{v^*-\mathcal{P}_{v_b,R}(v_\alpha(\rho_t))}dt \right . \\
        & \phantom{XXXXXXXXXXXXXXXXXXX}\;\;\,\left.  +2\sigma\sum_{i=1}^N \left(\normm{\overbar{V}^i_t-\vrho}\wedge M\right)\left(\frac{(\overbar{V}^i_t-\globmin)}{N\revised{\kappa}^2}\right)^{\top}d B^i_t \right\}.
    \end{aligned}
\end{equation}
Notice additionally that
\begin{equation}\label{eq:4888}
    \inner{\overbar{V}_t^i-v^*}{v^*-\mathcal{P}_{v_b,R}(v_\alpha(\rho_t))} \leq \normm{\overbar{V}_t^i-v^*}\normm{v^*-\mathcal{P}_{v_b,R}(v_\alpha(\rho_t))} \leq 2 R \normm{\overbar{V}_t^i-v^*}
\end{equation}
as $v^*$ and $\mathcal{P}_{v_b,R}(v_\alpha(\rho_t))$ belong to the same ball $B_R(v_b)$ around  $v_b$ of radius $R$.
Similarly, we can expand the coefficient of the second term.
According to the properties $\ff^{\prime}\in [0,1]$ and $\ff^{\prime\prime}\leq0$ we can bound it from above yielding
\begin{equation}\label{eq:488}
    \begin{aligned}
        &\frac{1}{2}\Gwn\sum_{i=1}^N\left(\ff^{\prime}\frac{2d}{N\revised{\kappa}^2}+\ff^{\prime\prime}\frac{4\normm{\overbar{V}^i_t-\globmin}^2}{N^2\revised{\kappa}^2}+\left(\ff^{\prime}\right)^2\frac{4\normm{\overbar{V}^i_t-\globmin}^2}{N^2\revised{\kappa}^4}\right)\left(\sigma\normm{\overbar{V}^i_t-v_\alpha\left(\rho_t\right)}\wedge M\right)^2\\
    &\quad\leq \Gwn\ff^{\prime}\frac{\sigma^2M^2d}{\revised{\kappa}^2}+\Gwn \left(\ff^{\prime}\right)^2\frac{2\sigma^2M^2}{N^2\revised{\kappa}^4}\sum_{i=1}^N\normmsq{\overbar{V}_t^i-v^*}\\
    &\quad\leq \Gwn\ff^{\prime}\frac{\sigma^2M^2d}{\revised{\kappa}^2}+\Gwn \ff^{\prime}\frac{2\sigma^2M^2}{N^2\revised{\kappa}^4}\sum_{i=1}^N\normmsq{\overbar{V}_t^i-v^*}.
    \end{aligned}
\end{equation}
By taking expectations in \eqref{eq:crux1} and combining it with \eqref{eq:477}, \eqref{eq:4888} and \eqref{eq:488}, we obtain
\begin{equation*}
 \begin{aligned}
				\frac{d}{dt}\Ep{\Gwn}
 &\leq \mathbb{E} \left[ \Gwn\ff^{\prime}\left(\frac{-2\lambda}{N\revised{\kappa}^2}\sum_{i=1}^N\normmsq{\overbar{V}_t^i-v^*}+\frac{4 R \lambda}{N\revised{\kappa}^2}\sum_{i=1}^N\normm{\overbar{V}_t^i-v^*}\right. \right.\\
&\quad\,+ \left.\left. \Gwn\ff^{\prime}\frac{\sigma^2M^2d}{\revised{\kappa}^2}+\Gwn \ff^{\prime}\frac{2\sigma^2M^2}{N^2\revised{\kappa}^4}\sum_{i=1}^N\normmsq{\overbar{V}_t^i-v^*} \right)\right]
			\end{aligned}
    \end{equation*}
Rearranging the former yields
		\begin{equation}\label{eq:44}
			\begin{aligned}
				\frac{d}{dt}\Ep{\Gwn}
				&\leq \Ep{\Gwn\ff^{\prime}\left(\left(\left(\frac{4\lambda R}{N\revised{\kappa}^2}\sum_{i=1}^N \normm{\overbar{V}^i_t-\globmin}\right)+\frac{\sigma^2M^2d}{\revised{\kappa}^2}\right)\right.\right.\\
                &\quad\,\left.\left.-\left(\frac{2\lambda}{N\revised{\kappa}^2}-\frac{2\sigma^2M^2}{N^2\revised{\kappa}^4}\right)\sum_{i=1}^N\normmsq{\overbar{V}^i_t-\globmin}\right)},
			\end{aligned}
		\end{equation}
		Since by Young's inequality, it holds $4R\normm{\overbar{V}^i_t-v^*}\leq 4R^2+\normmsq{\overbar{V}^i_t-v^*}$, we can continue Estimate~\eqref{eq:44} by
		\begin{equation} \label{eq:44444}
			\begin{aligned}
				\frac{d}{dt}\Ep{\Gwn}
				&\leq \Ep{\Gwn\ff^{\prime}\left(\frac{\sigma^2M^2d+\revised{4} \lambda R^2}{\revised{\kappa}^2}-\left(\frac{ \lambda}{N\revised{\kappa}^2}-\frac{2\sigma^2M^2}{N^2\revised{\kappa}^4}\right)\sum_{i=1}^N\normmsq{\overbar{V}^i_t-\globmin}\right)}\\
                &\leq \Ep{\Gwn\ff^{\prime}\left(-{A}\sum_{i=1}^N\normmsq{\overbar{V}^i_t-\globmin}+{B}\right)}
			\end{aligned}
		\end{equation}
		with $A:=\frac{\lambda}{N\revised{\kappa}^2}-\frac{2\sigma^2M^2}{N^2\revised{\kappa}^4}$ and $B:=\frac{\sigma^2M^2d+4\lambda R^2}{\revised{\kappa}^2}$.
		Now, if  $\sum_{i=1}^N\normmsq{\overbar{V}^i_t-\globmin}\geq ({B-1})/{A}$, we have 
		%\begin{equation*}
    	\begin{align*}
			&\Gwn\ff^{\prime}\left(-{A}\sum_{i=1}^N\normmsq{\overbar{V}^i_t-\globmin}+{B}\right)\leq 0,\\
            \intertext{while, if $\sum_{i=1}^N\normmsq{\overbar{V}^i_t-\globmin}\leq ({B-1})/{A}$, we have}
    	    &\Gwn\ff^{\prime}\left(-{A}\sum_{i=1}^N\normmsq{\overbar{V}^i_t-\globmin}+{B}\right)\leq Be^{\frac{B-1}{N\revised{\kappa}^2A}}.
		\end{align*}
        %\end{equation*}
		Thus the latter inequality always holds true
        %\begin{equation*}
		% 	{\Gwn\ff^{\prime}\left(-{A}\sum_{i=1}^N\normmsq{\overbar{V}^i_t-\globmin}+{B}\right)}\leq Be^{\frac{B-1}{N\revised{\kappa}^2A}},
		% \end{equation*}
		and consequently we have with \eqref{eq:44444}
		\begin{equation*}
			\begin{aligned}
				\frac{d}{dt}\Ep{\Gwn}\leq Be^{\frac{B-1}{N\revised{\kappa}^2A}},
			\end{aligned}	
		\end{equation*}
		which gives after integration
		\begin{equation*}
			\begin{aligned}
				\Ep{\Gwn}
                &\leq\Ep{\Gwno}+Be^{\frac{B-1}{N\revised{\kappa}^2A}}t\\
                &\leq \Ep{\exp\left(\frac{\sum_{i=1}^N\normm{\overbar{V}^i_0-\globmin}^2}{N\revised{\kappa}^2}\right)}+Be^{\frac{B-1}{N\revised{\kappa}^2A}}t.
			\end{aligned}
		\end{equation*}
		Letting $W\to\infty$, we eventually obtain
		\begin{equation}
			\Ep{\exp\left(\frac{\sum_{i=1}^N\normmsq{\bV^i_t-\globmin}}{N\revised{\kappa}^2}\right)}
            \leq \Ep{\exp\left(\frac{\sum_{i=1}^N\normm{\overbar{V}^i_0-\globmin}^2}{N\revised{\kappa}^2}\right)}+Be^{\frac{B-1}{N\revised{\kappa}^2A}}t<\infty,
		\end{equation}
		provided that $\Ep{\exp({\sum_{i=1}^N\normm{\overbar{V}^i_0-\globmin}^2}/{N\revised{\kappa}^2})}<\infty$.
  
        If $N\geq {(4 \sigma^2 M^2)}/{(\lambda \revised{\kappa}^2)}$ , we have 
		\begin{equation*}
			\frac{B-1}{N\revised{\kappa}^2A}\leq\frac{B}{N\revised{\kappa}^2A}=\frac{N(\sigma^2M^2d+4\lambda R^2)}{\lambda N\revised{\kappa}^2-2\sigma^2M^2}\leq C(\revised{\kappa},\lambda,\sigma,M,R,d).
		\end{equation*}
		Thus, $\gC$ is upper bounded and independent of $N$.
	\end{proof}

\begin{remark}
    \revised{The sub-Gaussianity of $\bV_t$ follows from Lemma~\ref{thm:2} by noticing that the statement can be applied in the setting $N=1$ when choosing $\kappa$ sufficiently large.}
\end{remark}
\begin{remark}
    In Lemma \ref{thm:2}, as the number of particles $N$ increases, the condition for $\revised{\kappa}$ to ensure $\gC<\infty$ becomes more relaxed.
    Specifically, the value of $\revised{\kappa}$ can be as small as one needs as $N$ increases.
    This phenomenon can be easily understood by considering the limit as $N$ approaches infinity.
    In this case, $\gC$ tends to $\sup_{t\in[0,T^*]}\exp(\Ep{\normmsq{\overbar{V}_t-v^*}}/\revised{\kappa}^2)$.
    Therefore, as one shows an upper bound on the second moment of $\overbar{V}_t$, it becomes evident that $\gC$ remains finite as $N$ tends to infinity.
\end{remark}
\noindent
With the help of Lemma \ref{thm:2}, we can now prove the following lemma.
\begin{lemma}\label{lem:332}
    Let $\CE \in \CC(\bbR^d)$ satisfy \ref{asp:11} and \ref{asp:33}.
    Then, for any $t\in [0,T^*]$, $M$ and $R$ with $R\geq \normm{v_b-v^*}$ finite, and $N$ satisfying $N\geq (16\alpha L_{\gamma}\sigma^2M^2)/\lambda$,
        we have 
	\begin{equation}
		\Ep{\normm{\vrho-{v}_{\alpha}(\bar{\rho}_t^{\revised{N}})}^2}\leq \frac{C_0}{N},
	\end{equation}
	where $C_0:=C_0(\lambda,\sigma,d,\alpha,L_{\gamma},L_u,T^*,R,v_b,v^*,M)$.
	\end{lemma}
			\begin{proof}
				Without loss of generality, we assume $\globmin=0$ and recall that we assumed $\underline{f}=0$ in the proofs as of Remark~\ref{remark:wlog}.
				%  since
				%\begin{equation}
				%\begin{aligned}
				%	&\Ep{\normm{\vrho-{v}_{\alpha}(\bar{\rho}_t)}^2}=\Ep{\normm{\frac{\frac{1}{N}\sum_{i=1}^N\bV_t^ie^{-\alpha f(\bV_t^i)}}{\frac{1}{N}\sum_{i=1}^Ne^{-\alpha f(\bV_t^i)}}-\frac{\int_{\mathbb{R}^d}v e^{-\alpha f(v)}d\rho_t(v)}{\int_{\mathbb{R}^d}e^{-\alpha f(v)}d\rho_t(v)}}^2}\\
				%	&\quad=\Ep{\normm{\frac{\frac{1}{N}\sum_{i=1}^N(\bV_t^i-\globmin)e^{-\alpha (f(\bV_t^i)-f(\globmin))}}{\frac{1}{N}\sum_{i=1}^Ne^{-\alpha (f(\bV_t^i)-f(\globmin))}}-\frac{\int_{\mathbb{R}^d}(v-\globmin) e^{-\alpha (f(v)-f(\globmin))}d\rho_t(v)}{\int_{\mathbb{R}^d}e^{-\alpha (f(v)-f(\globmin))}d\rho_t(v)}}^2}.
				%	\end{aligned}
			%\end{equation}
			We have
			\begin{equation} \label{eq:zsafbip}
				\begin{aligned}
					&\Ep{\normm{\vrho-{v}_{\alpha}(\bar{\rho}_t^{\revised{N}})}^2}=\Ep{\normm{\frac{\frac{1}{N}\sum_{i=1}^N\bV_t^ie^{-\alpha f(\overbarscript{V}_t^i)}}{\frac{1}{N}\sum_{i=1}^Ne^{-\alpha f(\overbarscript{V}_t^i)}}-\frac{\int_{\mathbb{R}^d}v e^{-\alpha f(v)}d\rho_t(v)}{\int_{\mathbb{R}^d}e^{-\alpha f(v)}d\rho_t(v)}}^2}\\
					&\quad\leq 2\Ep{\normmsq{\frac{1}{\frac{1}{N}\sum_{i=1}^Ne^{-\alpha f(\overbarscript{V}_t^i)}}\left({\frac{1}{N}\sum_{i=1}^N\bV_t^ie^{-\alpha f(\overbarscript{V}_t^i)}}-{\int_{\mathbb{R}^d}v e^{-\alpha f(v)}d\rho_t(v)}\right)}}\\
					& \quad\quad\,+2\Ep{\normm{\frac{v_{\alpha}(\rho_t)}{\frac{1}{N}\sum_{i=1}^Ne^{-\alpha f(\overbarscript{V}_t^i)}}\left(\frac{1}{N}\sum_{i=1}^Ne^{-\alpha f(\overbarscript{V}_t^i)}-\int_{\mathbb{R}^d}e^{-\alpha f(v)}d\rho_t(v)\right)}^{2}}\\
					&\quad\leq 2\Ep{\normmsq{e^{\alpha\frac{1}{N}\sum_{i=1}^Nf(\overbarscript{V}_t^i)}\left({\frac{1}{N}\sum_{i=1}^N\bV_t^ie^{-\alpha f(\overbarscript{V}_t^i)}}-{\int_{\mathbb{R}^d}v e^{-\alpha f(v)}d\rho_t(v)}\right)}}\\
					&\quad\quad\,+2\normmsq{\vrho}\Ep{\normm{e^{\alpha\frac{1}{N}\sum_{i=1}^Nf(\overbarscript{V}_t^i)}\left(\frac{1}{N}\sum_{i=1}^Ne^{-\alpha f(\overbarscript{V}_t^i)}-\int_{\mathbb{R}^d}e^{-\alpha f(v)}d\rho_t(v)\right)}^{2}}\\
					&\quad\leq 
					2T_1 T_2+2\normmsq{\vrho}T_1 T_3,
					%   \\&2\underbrace{\left(\Ep{{e^{4\alpha\frac{1}{N}\sum_{i=1}^Nf(\bV_t^i)}}}\right)^{\frac{1}{2}}}_{T_1}\underbrace{\left(\Ep{\normm{{\frac{1}{N}\sum_{i=1}^N\bV_t^ie^{-\alpha f(\bV_t^i)}}-{\int_{\mathbb{R}^d}v e^{-\alpha f(v)}d\rho_t(v)}}^4}\right)^{\frac{1}{2}}}_{T_2}\\
					%				&\quad +2\normmsq{\vrho}\left(\Ep{{e^{4\alpha\frac{1}{N}\sum_{i=1}^Nf(\bV_t^i)}}}\right)^{\frac{1}{2}}\underbrace{\left(\Ep{\normm{{\frac{1}{N}\sum_{i=1}^Ne^{-\alpha f(\bV_t^i)}}-{\int_{\mathbb{R}^d}e^{-\alpha f(v)}d\rho_t(v)}}^4}\right)^{\frac{1}{2}}}_{ T_3},
				\end{aligned}
			\end{equation}
			where we defined
			\begin{align*}
				T_1&:=\left(\Ep{{e^{4\alpha\frac{1}{N}\sum_{i=1}^Nf(\overbarscript{V}_t^i)}}}\right)^{\frac{1}{2}},\\
				T_2&:=\left(\Ep{\normm{{\frac{1}{N}\sum_{i=1}^N\bV_t^ie^{-\alpha f(\overbarscript{V}_t^i)}}-{\int_{\mathbb{R}^d}v e^{-\alpha f(v)}d\rho_t(v)}}^4}\right)^{\frac{1}{2}},\\
				 T_3&:=\left(\Ep{\normm{{\frac{1}{N}\sum_{i=1}^Ne^{-\alpha f(\overbarscript{V}_t^i)}}-{\int_{\mathbb{R}^d}e^{-\alpha f(v)}d\rho_t(v)}}^4}\right)^{\frac{1}{2}}.
			\end{align*}
			In the following, we upper bound the terms $T_1, T_2$ and $ T_3$ separately.
			Firstly, recall that by Lemma~\ref{thm:2} we have for $t\in [0,T^*]$ that
            \begin{equation}
				\Ep{\exp\left(\frac{\sum_{i=1}^N\normmsq{\bV_t^i}}{N\revised{\kappa}^2}\right)}\leq C_{\revised{\kappa}}<\infty,
			\end{equation}
            where $\gC$ only depends on $\revised{\kappa},\lambda,\sigma,d,R,M$ and $T^*$.
            With this,
			\begin{equation*}
				\begin{aligned}
					T_1^2 = \Ep{{\exp\left({4\alpha\frac{1}{N}\sum_{i=1}^Nf(\bV_t^i)}\right)}}&\leq \Ep{\exp\left(4\alpha\frac{1}{N}\sum_{i=1}^NL_{\gamma}\left(1+\normm{\bV_t^i}^{1+\gamma}\right)\right)}\\
					&\leq e^{4\alpha L_{\gamma}}\Ep{\exp\left(4\alpha L_{\gamma}\frac{1}{N}\sum_{i=1}^N\normm{\bV_t^i}^{1+\gamma}\right)}\\
					&\leq e^{8\alpha L_{\gamma}}\Ep{\exp\left(4\alpha L_{\gamma}\frac{1}{N}\sum_{i=1}^N\normm{\bV_t^i}^{2}\right)}\\
					&=e^{8\alpha L_{\gamma}}\Ep{\exp\left(\frac{1}{\revised{\kappa}^2}\frac{1}{N}\sum_{i=1}^N\normm{\bV_t^i}^{2}\right)} \\
					&\leq e^{8\alpha L_{\gamma}}\gC\!\!\mid_{\revised{\kappa}=\frac{1}{2\sqrt{\alpha L_{\gamma}}}},
				\end{aligned}
			\end{equation*}
            where we set $\revised{\kappa}^2= {1}/{(4\alpha L_{\gamma})}$ in the next-to-last step and where $N$ should satisfy $N\geq(16\alpha L_{\gamma}\sigma^2M^2)/\lambda$.
			%	\MF{I see here how A3 plays a crucial role and all depends on (63). We have faster than Gaussian concentration, then we could allow faster growth of $f$ ...}
			Secondly, we have 
			\begin{equation*}
				\begin{aligned}
					\Ep{\normm{{\frac{1}{N}\sum_{i=1}^N\bV_t^ie^{-\alpha f(\overbarscript{V}_t^i)}}-{\int_{\mathbb{R}^d}v e^{-\alpha f(v)}d\rho_t(v)}}^4}&=\frac{1}{N^4}\Ep{\sum_{i_1,i_2,i_3,i_4\in\{1,\dots,N\}}\inner{\overbar{Z}_t^{i_1}}{\overbar{Z}_t^{i_2}}\inner{\overbar{Z}_t^{i_3}}{\overbar{Z}_t^{i_4}}}\\
					&\leq \frac{4!L_u^4}{N^2},
				\end{aligned}
			\end{equation*}
			where $\left(\overbar{Z}_t^i:=\bV_t^ie^{-\alpha f(\overbarscript{V}_t^i)}-{\int_{\mathbb{R}^d}v e^{-\alpha f(v)}d\rho_t(v)}\right)_{i=1,\dots,N}$ are i.i.d.\@ and have zero mean.
			Thus,
			\begin{equation*}
				T_2 = \left(\Ep{\normm{{\frac{1}{N}\sum_{i=1}^N\bV_t^ie^{-\alpha f(\overbarscript{V}_t^i)}}-{\int_{\mathbb{R}^d}v e^{-\alpha f(v)}d\rho_t(v)}}^4}\right)^{\frac{1}{2}}\leq \frac{5L_u^2}{N}.
			\end{equation*}
			Similarly, we can derive 
			\begin{equation*}
				T_3 = \left(\Ep{\normm{{\frac{1}{N}\sum_{i=1}^Ne^{-\alpha f(\overbarscript{V}_t^i)}}-{\int_{\mathbb{R}^d}e^{-\alpha f(v)}d\rho_t(v)}}^4}\right)^{\frac{1}{2}}\leq\frac{5}{N}.
			\end{equation*}
			Collecting the bounds for the terms $T_1$, $T_2$ and $T_3$ and inserting them in \eqref{eq:zsafbip}, we obtain 
			%			\MF{I would make the boundedness of $\normm{v_{\alpha}(\rho_t)}$
				%				from Lemma \ref{prop:1}, Lemma \ref{lem:7}, Lemma \ref{prop:2} a bit more explicit for reader convenience.}
			\begin{equation} \label{eq:asdfinasdf}
				\begin{aligned}
					\Ep{\normm{\vrho-{v}_{\alpha}(\bar{\rho}_t)}^2}&\leq 10e^{6\alpha L_{\gamma}}\gC^{\frac{1}{2}}\!\!\mid_{\revised{\kappa}=\frac{1}{2\sqrt{\alpha L_{\gamma}}}}\left(L_u^2+\sup_{t\in[0,T^*]}\normmsq{\vrho}\right)\frac{1}{N}.
					%	&\leq \frac{C_0}{N}
				\end{aligned}
			\end{equation}
			Since by Lemmas \ref{prop:1}, \ref{lem:7} and \ref{prop:2},  we know that $\normm{v_{\alpha}(\rho_t)}$
				can be uniformly bounded by a constant depending on $\alpha,\lambda,\sigma,d,R,v_b,\globmin,M,L_{\nu}$ and $\nu$ (see in particular \Cref{eq:65} that combines the aforementioned lemmas),
            we can conclude \eqref{eq:asdfinasdf} with
            \begin{equation}
				\Ep{\normm{\vrho-{v}_{\alpha}(\bar{\rho}_t)}^2}\leq \frac{C_0}{N}
			\end{equation}
			for some constant $C_0$ depends on $\lambda,\sigma,d,\alpha,L_{\nu},\nu,L_{\gamma},L_u,T^*,R,v_b,\globmin$ and $M$.
			%		if $\gC^{\frac{1}{2}}\geq e^{\frac{1}{2}(4\alpha L_{\gamma})^{\frac{2}{1-\gamma}}\revised{\kappa}^{\frac{2(1+\gamma)}{1-\gamma}}}$, then we have
			%		\begin{equation}
				%			\Ep{\normm{\vrho-{v}_{\alpha}(\bar{\rho}_t)}^2}\leq 20e^{\alpha L_{\gamma}}\gC^{\frac{1}{2}}\left(L_u^2+\sup_{t\in[0,T^*]}\normmsq{\vrho}\right)\frac{1}{N},
				%		\end{equation}
			%		otherwise,we can decrease $\revised{\kappa}$, till $\gC^{\frac{1}{2}}\geq e^{\frac{1}{2}(2\alpha L_{\gamma})^{\frac{2}{1-\gamma}}\revised{\kappa}^{\frac{2(1+\gamma)}{1-\gamma}}}$ or till $\revised{\kappa}=\left(2\alpha L_{\gamma}\right)^{-\frac{1}{2}}$~(in this case $e^{\frac{1}{2}(2\alpha L_{\gamma})^{\frac{2}{1-\gamma}}\revised{\kappa}^{\frac{2(1+\gamma)}{1-\gamma}}}=e^{{\alpha L_{\gamma}}}$), so all in all, we have
			%		\begin{equation}
				%			\Ep{\normm{\vrho-{v}_{\alpha}(\bar{\rho}_t)}^2}\leq 20e^{2\alpha L_{\gamma}}\gC^{\frac{1}{2}}\left(L_u^2+\sup_{t\in[0,T^*]}\normmsq{\vrho}\right)\frac{1}{N}.
				%		\end{equation}
		\end{proof}
		
		\subsubsection{Upper Bound for the Third Term in \eqref{eq:error_decomposition}}
  
		In this section, we bound Term $III$ of the error decomposition~\eqref{eq:error_decomposition}.
        Before stating the main result of this section, Proposition \ref{thm:1}, we first need to provide two auxiliary lemmas, Lemma~\ref{lem:4} and Lemma~\ref{prop:1}.
		\begin{lemma}\label{lem:4}
			Let $R,M\in (0,\infty)$.
            \revised{Then it holds}
			\begin{equation}
				\begin{aligned}
					\frac{d}{dt}\Ep{ \normmsq{\bV_t-\globmin}}
                    &\leq -\lambda\Ep{ \normmsq{\bV_t-\globmin}}\\
                    &\quad\,+\lambda\left(\normmsq{\rhoM-\globmin}+\normmsq{\vrho-\rhoM}\right)\\
					&\quad\,+\sigma^2M^2d.
				\end{aligned}
			\end{equation}
			If further $\lambda\geq 2\sigma^2d$, we have
			\begin{equation}
            \begin{aligned}
				\frac{d}{dt}\Ep{ \normmsq{\bV_t-\globmin}}
                &\leq -\lambda\Ep{ \normmsq{\bV_t-\globmin}}\\
                &\quad\,+\lambda\left(\normmsq{\rhoM-\globmin}+\normmsq{\vrho-\rhoM}\right).
            \end{aligned}
			\end{equation}
		\end{lemma}
		\begin{proof}
			%		We first prove \Cref{eq:48} and \Cref{eq:49}.
			%		By  It\^o formula, we have
			%		\begin{equation}\label{eq:51}
				%			\begin{aligned}
					%				d \normmsq{\bV_t-\globmin}&=2\left(\bV_t-\globmin\right)^{\top}d \bV_t+\sigma^2 d\normmsq{\bV_t-\vrho}dt\\
					%				&=-2\lambda\inner{\bV_t-\globmin}{\bV_t-v_{\alpha}(\rho_t)}dt+2\sigma\normm{\bV_t-\vrho}\left(\bV_t-\globmin\right)^{\top}dB_t\\
					%				&\quad+\sigma^2 d\normmsq{\bV_t-\vrho}dt\\
					%				&= -\lambda {\normmsq{{\bV}_t-\globmin}}+\lambda \normmsq{\vrho-\globmin}-\left(\lambda-\sigma^2d\right){\normmsq{\bV_t-v_{\alpha}(\rho_t)}}\\
					%				&\quad +2\sigma\normm{\bV_t-\vrho}\left(\bV_t-\globmin\right)^{\top}dB_t,
					%			\end{aligned}
				%		\end{equation}
			%		the last step is due to the equality
			%		\begin{equation}
				%			2\inner{\bV_t-\globmin}{\bV_t-v_{\alpha}(\rho_t)}=\normmsq{\bV_t-\globmin}+\normmsq{\bV_t-v_{\alpha}(\rho_t)}-\normmsq{\globmin-v_{\alpha}(\rho_t)}.
				%		\end{equation}
			%		Take expectation on both sides of \Cref{eq:51}, we have \Cref{eq:48}; let $d\geq \sigma^2d$, we have \Cref{eq:49}.
			%		
			%		The second part proof is of the similar nature.
			By It\^o's formula, we have
			\begin{equation*}
				\begin{aligned}
					d \normmsq{\bV_t-\globmin}&=2\left(\bV_t-\globmin\right)^{\top}d \bV_t+\sigma^2 d\left(\normmsq{\bV_t-\vrho}\wedge M^2\right)dt\\
					&=-2\lambda\inner{\bV_t-\globmin}{\bV_t-\rhoM}dt+2\sigma\left(\normm{\bV_t-\vrho}\wedge M\right)\left(\bV_t-\globmin\right)^{\top}dB_t\\
					&\quad\,+\sigma^2 d\left(\normmsq{\bV_t-\vrho}\wedge M^2\right)dt\\
					&=-\lambda \left[\normmsq{\bV_t-\globmin}+\normmsq{\bV_t-\rhoM}-\normmsq{\rhoM-\globmin}\right]dt\\
					&\quad\,+2\sigma\left(\normm{\bV_t-\vrho}\wedge M\right)\left(\bV_t-\globmin\right)^{\top}dB_t+\sigma^2 d\left(\normmsq{\bV_t-\vrho}\wedge M^2\right)dt,
				\end{aligned}
			\end{equation*}
			which, after taking the expectation on both sides, yields
			\begin{equation}\label{eq:55}
				\begin{aligned}
					\frac{d}{dt}\Ep{ \normmsq{\bV_t-\globmin}}&= -\lambda\Ep{ \normmsq{\bV_t-\globmin}}\\&\quad\,+\lambda \normmsq{\rhoM-\globmin}-\lambda\Ep{\normmsq{\bV_t-\rhoM}}\\
					&\quad\,+\sigma^2 d\Ep{\normmsq{\bV_t-\vrho}\wedge M^2}.
				\end{aligned}
			\end{equation}
			For the term $\Ep{\normmsq{\bV_t-\rhoM}}$, we notice that 
			\begin{equation*}
				\begin{aligned}
					\Ep{\normmsq{\bV_t-\rhoM}}&=\Ep{ \normmsq{\bV_t-\vrho}}+\Ep{\normmsq{\vrho-\rhoM}}\\
					&\quad\,+2\Ep{\inner{\bV_t-\vrho}{\vrho-\rhoM}}\\
					&\geq \Ep{ \normmsq{\bV_t-\vrho}}+\Ep{\normmsq{\vrho-\rhoM}}\\
					&\quad\,-\left(\frac{1}{2}\Ep{\normmsq{\bV_t-\vrho}}+2\Ep{\normmsq{\vrho-\rhoM}}\right)\\
					&=\frac{1}{2}\Ep{\normmsq{\bV_t-\vrho}}-\Ep{\normmsq{\vrho-\rhoM}},
				\end{aligned}
			\end{equation*}
			which, inserted into \Cref{eq:55}, allows to derive
			\begin{equation*}
				\begin{aligned}
					\frac{d}{dt}\Ep{ \normmsq{\bV_t-\globmin}}&\leq -\lambda\Ep{ \normmsq{\bV_t-\globmin}}\\&\quad\,+\lambda\left(\normmsq{\rhoM-\globmin}+\normmsq{\vrho-\rhoM}\right)\\
					&\quad\,-\frac{1}{2}\lambda\Ep{\normmsq{\bV_t-\vrho}}+\sigma^2 d\left(\normmsq{\bV_t-\vrho}\wedge M^2\right).
				\end{aligned}
			\end{equation*}
			From this we get for any $\lambda$ and $\sigma$ that
			\begin{equation}
				\begin{aligned}
					\frac{d}{dt}\Ep{ \normmsq{\bV_t-\globmin}}&\leq -\lambda\Ep{ \normmsq{\bV_t-\globmin}}\\&\quad\,+\lambda\left(\normmsq{\rhoM-\globmin}+\normmsq{\vrho-\rhoM}\right)\\
					&\quad\,+\sigma^2M^2d.
				\end{aligned}
			\end{equation}
			as well as
			\begin{equation}\label{eq:62}
				\begin{aligned}
					\frac{d}{dt}\Ep{ \normmsq{\bV_t-\globmin}}	&\leq -\lambda\Ep{ \normmsq{\bV_t-\globmin}}\\&\quad\,+\lambda\left(\normmsq{\rhoM-\globmin}+\normmsq{\vrho-\rhoM}\right)\\
					&\quad\,+\left(-\frac{1}{2}\lambda+\sigma^2d\right)\Ep{\normmsq{\bV_t-\vrho}}.
				\end{aligned}
			\end{equation}
			If $\lambda\geq 2\sigma^2d$, by \Cref{eq:62}, we get
			\begin{equation}
			\begin{aligned}
                \frac{d}{dt}\Ep{ \normmsq{\bV_t-\globmin}}&\leq -\lambda\Ep{ \normmsq{\bV_t-\globmin}}\\&\quad\,+\lambda\left(\normmsq{\rhoM-\globmin}+\normmsq{\vrho-\rhoM}\right).
            \end{aligned}
			\end{equation}
		\end{proof}
		\begin{remark}
			When $R=M=\infty$, we can show
			\begin{equation*}
            \begin{aligned}%\label{eq:48}
				\frac{d}{dt}\Ep{\normmsq{{\bV}_t-\globmin}}&= -\lambda \Ep{\normmsq{{\bV}_t-\globmin}}\\&\quad\,+\lambda \normmsq{\vrho-\globmin}-\left(\lambda-\sigma^2d\right)\Ep{\normmsq{\bV_t-v_{\alpha}(\rho_t)}}.
			\end{aligned}
            \end{equation*}
			If further $\lambda\geq \sigma^2d$, we have
			\begin{equation*}%\label{eq:49}
				\frac{d}{dt}\Ep{\normmsq{{\bV}_t-\globmin}}\leq -\lambda \Ep{\normmsq{{\bV}_t-\globmin}}+\lambda \normmsq{\vrho-\globmin}.
			\end{equation*}
           This differs from \cite[Lemma 18]{fornasier2021consensus}.
		\end{remark}
        \noindent
		The next result is a quantitative version of the Laplace principle as established in \cite[Proposition 21]{fornasier2021consensus}.
		\begin{lemma}\label{prop:1}
			For any $r>0$, define $f_r:= \sup _{v \in B_r\left(v^*\right)} f(v)$. Then, under the inverse continuity condition \ref{asp:22}, for any $r \in\left(0, R_0\right]$ and $q>0$ such that $q+f_r \leq f_{\infty}$, it holds
			\begin{equation}\label{eq:58}
				\left\|v_\alpha(\rho)-\globmin\right\|_2 \leq \frac{\left(q+f_r\right)^\nu}{L_{\nu}}+\frac{\exp (-\alpha q)}{\rho\left(B_r\left(\globmin\right)\right)} \int\normm{v-\globmin} d \rho(v)
			\end{equation}
		\end{lemma}
		%	\begin{proof}
			%		We only sketch the proof here, for more details, please check \cite[Proposition 21.]{fornasier2021consensus}. Actually, the proof is quite direct:
			%		\begin{equation}\label{eq:59}
				%			\begin{aligned}
					%				\normm{v_{\alpha}(\rho)-\globmin}&=\normm{\int_{\mathbb{R}^d}\frac{\left(v-\globmin\right)\omega_{\alpha}(v)}{\normm{\omega_{\alpha}}_{L_1(\rho)}}d\rho(v)}\\
					%				&\leq\normm{\int_{B_{\tilde{r}}(\globmin)}\frac{\left(v-\globmin\right)\omega_{\alpha}(v)}{\normm{\omega_{\alpha}}_{L_1(\rho)}}d\rho(v)}+\normm{\int_{\left(B_{\tilde{r}}(\globmin)\right)^c}\frac{\left(v-\globmin\right)\omega_{\alpha}(v)}{\normm{\omega_{\alpha}}_{L_1(\rho)}}d\rho(v)},\quad \forall ~\tilde{r}\geq 0,
					%			\end{aligned}
				%		\end{equation}
			%		then last two terms in \Cref{eq:59} can be separately upper bounded by the two terms in the right hand side of \Cref{eq:58}, by the properties of $f$ from \Cref{asp:2}, for some $\tilde{r}$ slightly bigger than $r$. 
			%	\end{proof}
		%	\begin{remark}The above lemma builds a quantitative bound for $\normm{v_{\alpha}(\rho)-\globmin}$. However by the above lemma,
			%		to have $\normm{v_{\alpha}(\rho)-\globmin}$ small, we need $\alpha=O\left(d\right)$, since generally we have $\rho\left(B_r(\globmin)\right)=\mathcal{O}\left(r^d\right)$.
			%	\end{remark}

		\noindent
		With the above preparation, we can now upper bound Term~$III$.
        We have by Jensen's inequality
		\begin{equation}
			III=\Ep{\normmsq{\frac{1}{N} \sum_{i=1}^N \bV_{T^*}^i-\globmin}}\leq \frac{1}{N}\sum_{i=1}^N\Ep{\normmsq{ \bV_{T^*}^i-\globmin}},
		\end{equation}
		i.e., it is enough to upper bound $\Ep{\normmsq{ \bV_{T^*}-\globmin}}$, which is the content of the next statement.

		\begin{proposition}\label{thm:1}
			Let $\CE \in \CC(\bbR^d)$ satisfy \ref{asp:11}, \ref{asp:22} and \ref{asp:33}.
            Moreover, let $\rho_0\in\CP_4(\bbR^d)$ with $\globmin\in\supp(\rho_0)$.
			Fix any $\epsilon \in (0,W_2^2(\rho_0,\delta_{\globmin}))$ and define the time horizon
			$$T^*:=\frac{1}{\lambda}\log\left(\frac{2W_2^2(\rho_0,\delta_{\globmin})}{\epsilon}\right).$$
            Moreover, let $R\in(\normm{v_b-\globmin}+\sqrt{\epsilon/2},\infty)$, $M\in(0,\infty)$ and $\lambda,\sigma>0$ be such that $\lambda\geq 2\sigma^2d$ or $\sigma^2M^2d=\mathcal{O}(\epsilon)$.
            Then we can choose $\alpha$ sufficiently large, depending on $\lambda,\sigma,d,T^*,R,v_b,M,\epsilon$ and properties of $f$,
            such that $ \Ep{\normmsq{\bV_{T^*}-\globmin}} =\mathcal{O}(\epsilon)$.
		\end{proposition}
		
		\begin{proof}
			We only prove the case $\lambda\geq 2\sigma^2d$ in detail.
            The case $\sigma^2M^2d=\mathcal{O}(\epsilon)$ follows similarly.
			
			According to Lemmas~\ref{prop:1} and \ref{prop:2}, we have
			\begin{equation}\label{eq:65}
				\begin{aligned}
					\left\|\vrho-\globmin\right\|_2 &\leq \frac{\left(q+f_r\right)^\nu}{L_{\nu}}+\frac{\exp (-\alpha q)}{\rho_t\left(B_r\left(\globmin\right)\right)} \Ep{\normm{\bV_t-\globmin}} \\
					&\leq \frac{\left(q+f_r\right)^\nu}{L_{\nu}}+{\exp(-\alpha q)}\cin\cfm,
				\end{aligned}
			\end{equation}
			where $\cin:=(\exp{q'T^*})/{C_{4}}<\infty$, $q'$ and $C_4$ are from Lemma~\ref{prop:2}, and where, as of Lemma~\ref{lem:7}, $\cfm:=\sup_{[0,T^*]}\Ep{\normm{\overbar{V}_t-v^*}}<\infty$.
            In what follows, let us deal with the two terms on the right-hand side of \eqref{eq:65}.
            For the term ${\left(q+f_r\right)^\nu}/{L_{\nu}}$, let $q=f_r$. Then by \ref{asp:22} and \ref{asp:33}, we can choose proper $r$, such that $\revised{2(L_{\nu} r)^{{1}/{\nu}}}\leq2f_r\leq f_{\infty}$.
            Further by \ref{asp:33}, we have
			\begin{equation*}
				\frac{\left(q+f_r\right)^\nu}{L_{\nu}}=\frac{(2f_r)^{\nu}}{L_{\nu}}\leq 	\frac{(2L_{\gamma})^{\nu}r^{(1+\gamma)\nu}}{L_{\nu}},
			\end{equation*}
			so if
            \begin{equation*}
				r< r_0:=\min\left\{\left(\frac{\epsilon}{8}\right)^{\frac{1}{2(1+\gamma)\nu}}\left(\frac{L_{\nu}}{(2L_{\gamma})^{\nu}}\right)^{\frac{1}{(1+\gamma)\nu}},\sqrt{\frac{\epsilon}{2}}\right\},
			\end{equation*}  we can bound
			\begin{equation*}
				\frac{\left(q+f_r\right)^\nu}{L_{\nu}}=\frac{(2f_r)^{\nu}}{L_{\nu}}\leq \frac{\sqrt{\epsilon}}{2\sqrt{2}}.
			\end{equation*} 
			For term ${\exp(-\alpha q)}\cin\cfm$, we can choose $\alpha$ large enough such that 
			\begin{equation*}
				{\exp(-\alpha q)}\cin\cfm\leq \frac{\sqrt{\epsilon}}{2\sqrt{2}}.
			\end{equation*}
			%		for example, we can choose $\alpha> \alpha_0:=\frac{\log\left(\frac{2\sqrt{2}\cin\cfm }{\sqrt{\epsilon}}\right)}{L_{\nu} r^{\frac{1}{\nu}}}.$
			%		
			With these choices of $r$ and $\alpha$ and by integrating them into \Cref{eq:65}, we obtain
			\begin{equation*}%\label{eq:64}
				\normmsq{\vrho-\globmin}<\frac{\epsilon}{2},
			\end{equation*}
			for all $t\in[0,T^*]$, and thus
			\begin{equation*}
				\normm{v_{\alpha}(\rho_t)-v_b}\leq \normm{v_{\alpha}(\rho_t)-\globmin}+\normm{\globmin-v_b}\leq \sqrt{\frac{\epsilon}{2}}+\normm{\globmin-v_b}\leq R.
			\end{equation*}
			Consequently, by Lemma \ref{lem:4}, we have
			\begin{equation*}%\label{eq:69}
			\begin{aligned}
			    \frac{d}{dt}\Ep{ \normmsq{\bV_t-\globmin}}&\leq-\lambda\left(\Ep{\normmsq{\bV_t-\globmin}}-\normmsq{\vrho-\globmin}\right)\\&\leq -\lambda\left(\Ep{\normmsq{\bV_t-\globmin}}-\frac{\epsilon}{2}\right),
			\end{aligned}
			\end{equation*}
			since now $\mathcal{P}_{v_b,R}(v_{\alpha}(\rho_t))=v_{\alpha}(\rho_t)$. Finally by Gr\"onwall's inequality, $\Ep{\normmsq{\bV_{T^*}-\globmin}}\leq \epsilon$.
				\end{proof}
				
				\begin{lemma}\label{lem:7}
                Let $\normm{v_b-\globmin}<R<\infty$ and $0<M<\infty$. \revised{Then it holds} 
					\begin{equation}
						\sup_{t\in[0,T^*]}\Ep{\normm{\bV_t-\globmin}}\leq \sqrt{\max\left\{\Ep{ \normmsq{\bV_0-\globmin}},\lambda R^2+\sigma^2M^2d\right\}}.
					\end{equation}
				\end{lemma}

				\begin{proof}
					By \Cref{eq:55} we have
					\begin{equation*}
						\begin{aligned}
							\frac{d}{dt}\Ep{ \normmsq{\bV_t-\globmin}}&\leq -\lambda\Ep{ \normmsq{\bV_t-\globmin}}\\&\quad\,+\lambda \normmsq{\rhoM-\globmin}-\lambda\Ep{\normmsq{\bV_t-\rhoM}}\\
							&\quad\,+\sigma^2 d\Ep{\normmsq{\bV_t-\vrho}\wedge M^2}\\
							&\leq -\lambda\Ep{ \normmsq{\bV_t-\globmin}}+\lambda R^2+\sigma^2M^2d,
						\end{aligned}
					\end{equation*}
					yielding 
					\begin{equation*}
						\Ep{ \normmsq{\bV_t-\globmin}}\leq\max\left\{\Ep{ \normmsq{\bV_0-\globmin}},\lambda R^2+\sigma^2M^2d\right\},
					\end{equation*}
					after an application of Gr\"onwall's inequality for any $t\geq 0$.
				\end{proof}

				\begin{lemma}\label{prop:2}
                    For any $M\in (0,\infty)$, $\tau\geq1$, $r>0$ and $R\in(\normm{v_b-\globmin}+r,\infty)$ it holds
					\begin{equation*}
						\begin{aligned}
							\rho_t\left(B_r\left(v^*\right)\right) & \geq C_4\exp (-q' t)>0,
						\end{aligned}
					\end{equation*}
					where \begin{eqnarray*}
						&C_4:=\int_{B_r(\globmin)}1+(\tau-1)\normm{\frac{v-v^*}{r}}^{\tau}-\tau\normm{\frac{v-v^*}{r}}^{\tau-1} d \rho_0(v)
					\end{eqnarray*} and where $q'$ depends on $\tau,\lambda,\sigma,d,r,R,v_b$ and $M$.
				\end{lemma}
				%\begin{remark}\label{rmk:5}
				%In the above lemma, it is possible to relax the condition $R\geq \normm{v_b-\globmin}+ 4r, M\geq R+\normm{v_b-\globmin}+1$ to a less stringent requirement of $R> \normm{v_b-\globmin}$ and $M> R+\normm{v_b-\globmin}$. Consequently, the formula for $q'$ can be appropriately derived within the subsequent proof, following this adjustment.
				%\end{remark}
				\begin{proof}
                    Recall that the law~$\rho_t$ of $\bV_t$ satisfies the Fokker-Planck equation
					\begin{equation*}
						\partial_t \rho_t=\lambda \operatorname{div}\left(\left(v-\rhoM\right)\rho_t\right)+\frac{\sigma^2}{2} \Delta\left(\left(\left\|v-v_\alpha\left(\rho_t\right)\right\|^2\wedge M^2\right) \rho_t\right).
					\end{equation*}
					Let us first define for $\tau\geq1$ the test function 
					\begin{equation}
						\phi_r^{\tau}(v):=\begin{cases}
							1+(\tau-1)\normm{\frac{v}{r}}^{\tau}-\tau\normm{\frac{v}{r}}^{\tau-1},& \normm{v}\leq r,\\
							0,& \text{else},
						\end{cases}
					\end{equation}
					for which it is easy to verify that $\phi_r^{\tau}\in \mathcal{C}_c^1(\mathbb{R}^d,[0,1])$.
					Since $\mathrm{Im}\,\phi_r^{\tau}\subset [0,1]$, we have $\rho_t(B_r(\globmin))\geq \int_{B_r(\globmin)}\phi_r^{\tau}(v-\globmin)\,d\rho_t(v)$.
                    To lower bound $\rho_t(B_r(\globmin))$, it is thus sufficient to establish a lower bound on $\int_{B_r(\globmin)}\phi_r^{\tau}(v-\globmin)\,d\rho_t(v)$.
                    By Green's formula
					\begin{equation*}
						\begin{aligned}
							&\frac{d}{dt}\int_{B_r(\globmin)}\phi_r^{\tau}(v-\globmin)\,d\rho_t(v)=-\lambda\int_{B_r(\globmin)}\inner{v-\rhoM}{\nabla\phi_r^{\tau}(v-\globmin)}d\rho_t(v)\\
							&\quad\,\quad+\frac{\sigma^2}{2}\int_{B_r(\globmin)}\left(\normmsq{v-\vrho}\wedge M^2\right)\Delta\phi_r^{\tau}(v-\globmin)d\rho_t(v)\\
							&\quad\,=\tau(\tau-1)\int_{B_r(\globmin)}\frac{\normm{v-\globmin}^{\tau-3}}{r^{\tau-3}}\Bigg(\!\left(1-\frac{\normm{v-\globmin}}{r}\right)\bigg(\lambda\inner{\frac{v-\rhoM}{r}}{\frac{v-\globmin}{r}}\\
							&\quad\,\quad-\frac{\sigma^2}{2}\left(d+\tau-2\right)\frac{\normmsq{v-\vrho}\wedge M^2}{r^2}\bigg)+\frac{\sigma^2}{2}\frac{\normmsq{v-\vrho}\wedge M^2}{r^2}\Bigg)d\rho_t(v).		
						\end{aligned}
					\end{equation*}
					For simplicity, let us abbreviate
					\begin{equation*}
						\begin{aligned}
							\Theta:&=\left(1-\frac{\normm{v-\globmin}}{r}\right)\bigg(\lambda\inner{\frac{v-\rhoM}{r}}{\frac{v-\globmin}{r}}\\
							&\quad-\frac{\sigma^2}{2}\left(d+\tau-2\right)\frac{\normmsq{v-\vrho}\wedge M^2}{r^2}\bigg)
							+\frac{\sigma^2}{2}\frac{\normmsq{v-\vrho}\wedge M^2}{r^2}.
							%				&\geq \left(1-r_v\right)\left[\frac{\lambda}{2}\left(r_v^2+(r_v-r_{\mathcal{P}})^2-r_{\mathcal{P}}^2\right)-\frac{\sigma^2}{2}\left(d+\tau-2\right)\left((r_v+r_{\rho})^2\wedge r_M^2\right)\right]
							%				\\
							%				&\quad+\frac{\sigma^2}{2}\left((r_v-r_{\rho})^2\wedge r_M^2\right),
						\end{aligned}
					\end{equation*}
					We can choose $\epsilon_1$ small enough, depending on $\tau$ and $d$, such that when ${\normm{v-\globmin}}/{r}> 1-\epsilon_1$, we have
					\begin{equation*}
						\begin{aligned}
							\Theta&\revised{=\left(1-\frac{\normm{v-\globmin}}{r}\right)\lambda\inner{\frac{v-\rhoM}{r}}{\frac{v-\globmin}{r}}}\\
							&\revised{\quad\,+\bigg(\frac{\sigma^2}{2}-\left(1-\frac{\normm{v-\globmin}}{r}\right)\frac{\sigma^2}{2}(d+\tau-2)\bigg)\frac{\normmsq{v-\vrho}\wedge M^2}{r^2}}\\
							&\geq \left(1-\frac{\normm{v-\globmin}}{r}\right)\lambda\inner{\frac{v-\rhoM}{r}}{\frac{v-\globmin}{r}}
							+\frac{\sigma^2}{3}\frac{\normmsq{v-\vrho}\wedge M^2}{r^2},
						\end{aligned}
					\end{equation*}
				    \revised{where the last inequality works if ${\normm{v-\globmin}}/{r}\geq 1-1/(6(d+\tau-2))$.}
        
					If $v_{\alpha}(\rho_t)\not\in B_{R}(v_b)$, we have $\abs{\inner{v-\rhoM}{v-\globmin}}/r^2\leq C(r,R,v_b)$ and, since $R>\normm{v_b-\globmin}+r$, $({\normmsq{v-\vrho}\wedge M^2})/{r^2}\geq C(r,M,R,v_b)$,  which allows to choose $\epsilon_2$ small enough, depending on $\lambda,r,\sigma,R,v_b$ and $M$, such that $\Theta>0$ when ${\normm{v-\globmin}}/{r}> 1-\min\{\epsilon_1,\epsilon_2\}$. 
                    
                    If $v_{\alpha}(\rho_t)\in B_{R}(v_b)$ and $\normm{v-\vrho}\leq M$, we have by Lemma \ref{lem:333}
					\begin{equation*}
						\begin{aligned}
							\Theta&\geq \left(1-\frac{\normm{v-\globmin}}{r}\right)\lambda\inner{\frac{v-v_{\alpha}(\rho_t)}{r}}{\frac{v-\globmin}{r}}
							+\frac{\sigma^2}{3}\frac{\normmsq{v-\vrho}}{r^2}\\
							&=\left(\frac{\sigma^2}{3}+\left(1-\frac{\normm{v-\globmin}}{r}\right)\lambda\right)\frac{\normmsq{v-\globmin}}{r^2}+\frac{\sigma^2}{3}\frac{\normmsq{v_{\alpha}(\rho_t)-\globmin}}{r^2}\\
							&\quad\,-\left(\frac{2\sigma^2}{3}+\left(1-\frac{\normm{v-\globmin}}{r}\right)\lambda\right)\inner{\frac{v_{\alpha}(\rho_t)-\globmin}{r}}{\frac{v-\globmin}{r}}\\
							&\geq 0,
						\end{aligned}
					\end{equation*}
                    when $\normm{v-\globmin}/r\in \left[1-2\sigma^2/(3\lambda),1\right]$.
					
                    If $v_{\alpha}(\rho_t)\in B_{R}(v_b)$ and $\normm{v-\vrho}> M$, we have
					\begin{equation*}
						\Theta\geq \left(1-\frac{\normm{v-\globmin}}{r}\right)C(\lambda,r,R,v_b)+\frac{\sigma^2}{3}M^2,
					\end{equation*}
					i.e., we can choose $\epsilon_3$ small enough, depending on $\lambda,r,\sigma,R,v_b$ and $M$, such that $\Theta\geq0$ when ${\normm{v-\globmin}}/{r}> 1-\min\{\epsilon_1,\epsilon_2,\epsilon_3,{2\sigma^2}/{3\lambda}\}$.
     
                    Combining the cases from above, we conclude that $\Theta\geq 0$ when ${\normm{v-\globmin}}/{r}\geq 1-\min\{\epsilon_1,\epsilon_2,\epsilon_3,{2\sigma^2}/{3\lambda}\}$.
					On the other hand, when ${\normm{v-\globmin}}/{r}\leq 1-\min\{\epsilon_1,\epsilon_2,\epsilon_3,{2\sigma^2}/{3\lambda}\}$, we have
					\begin{equation*}
						\begin{aligned}
							\tau(\tau-1)\frac{\normm{v-\globmin}^{\tau-3}}{r^{\tau-3}}\Theta
							&=\tau(\tau-1)\frac{\normm{v-\globmin}^{\tau-3}}{r^{\tau-3}}\frac{\Theta}{\phi_r^{\tau}(v)}\phi_r^{\tau}(v-\globmin)\geq -C_5\phi_r^{\tau}(v-\globmin)
						\end{aligned}
					\end{equation*}
					for some constant $C_5$ depending on $r,R,M,v_b,\lambda,\sigma,d$ and $\tau$, since $\abs{\Theta}$ is upper bounded and $\phi_r^{\tau}(v-v^*)\geq \phi_r^{\tau}((1-\min\{\epsilon_1,\epsilon_2,\epsilon_3,{2\sigma^2}/{3\lambda}\})r)>0$ for any $v$ satisfies ${\normm{v-\globmin}}/{r}\leq 1-\min\{\epsilon_1,\epsilon_2,\epsilon_3,{2\sigma^2}/{3\lambda}\}$.
					
					All in all we have
					\begin{equation*}
						\begin{aligned}
							\frac{d}{dt}\int_{B_r(\globmin)}\phi_r^{\tau}(v-v^*)\,d\rho_t(v)\geq-q'\int_{B_r(\globmin)} \phi_r^{\tau}(v-v^*)\,d\rho_t(v),
						\end{aligned}
					\end{equation*}
					where $q':=\max\{C_5,0\}$.
					By Gr\"onwall's inequality, we thus have
					\begin{equation*}
						\rho_t(B_r(\globmin))\geq \int_{B_r(\globmin)}\phi_r^{\tau}(v-v^*)\,d\rho_t(v)\geq e^{-q't}\int_{B_r(\globmin)}\phi_r^{\tau}(v-v^*)\,d\rho_0(v),
					\end{equation*}
                   which concludes the proof.
				\end{proof}
				\begin{lemma}
					\label{lem:333} Let $a,b>0$. Then we have
					\begin{equation*}
						(a+b(1-x))x^2+ay^2-(2a+b(1-x))xy\geq 0,
					\end{equation*}
					for any $x\in [1-{2a}/{b},1]\cap(0,\infty)$ and $y\geq0$.
				\end{lemma}
				\begin{proof}
					For $y=0$, this is true.
                    For $y>0$, divide both side by $ay^2$ and denote $c={b}/{a}$.
                    Then the lemma is equivalent to showing $(1+c(1-x))\left(x/y\right)^2-(2+c(1-x))x/y+1\geq 0$,
					i.e., it is enough to show $\min_{r\geq 0}\,(1+c(1-x))r^2-(2+c(1-x))r+1\geq 0$,
					when $x\in [1-{2}/{c},1]$. We have
					\begin{equation*}
						\arg\min_{r}\, (1+c(1-x))r^2-(2+c(1-x))r+1 = \frac{2+c(1-x)}{2+2c(1-x)},
					\end{equation*}
					and thus 
					\begin{equation*}
						\begin{aligned}
							&\min_{r\geq 0}\,(1+c(1-x))r^2-(2+c(1-x))r+1\\
							&\quad =(1+c(1-x))\left(\frac{2+c(1-x)}{2+2c(1-x)}\right)^2-(2+c(1-x))\frac{2+c(1-x)}{2+2c(1-x)}+1\\
							&\quad =-\frac{1}{2}\frac{(2+c(1-x))^2}{2+2c(1-x)}+1\geq 0,
						\end{aligned}
					\end{equation*}
					when $x\in [1-{2}/{c},1]$.
                    This finishes the proof.
					%since $4+4z-(2+z)^2\geq 0$ when $z\in[0,2]$.
				\end{proof}

\section{Numerical Experiments} \label{sec:numerics}
				
	In this section we numerically demonstrate the benefit of using CBO with truncated noise.
    For isotropic \cite{pinnau2017consensus,carrillo2018analytical,fornasier2021consensus} and anisotropic noise \cite{carrillo2019consensus,fornasier2021convergence},
	we compare the CBO method with truncation $M=1$ to standard CBO for several benchmark problems in optimization, which are summarized in Table~\ref{table:benchmark}.    
    \newcommand\xrowht[2][0]{\addstackgap[.5\dimexpr#2\relax]{\vphantom{#1}}}
    \begin{table}[H]
    	\centering
      	%\rowcolors{2}{}{gray!10}
    	\begin{tabular}{cccc}
        	\toprule
        	\bf Name & \bf Objective function $f$  & $v^*$ & $\minobj$ \\ \midrule \xrowht{16pt}
        	Ackley & $-20 \exp \left(-0.2 \sqrt{\frac{1}{d} \sum_{i=1}^dv_i^2}\right)-\exp \left(\frac{1}{d} \sum_{i=1}^d \cos \left(2 \pi v_i\right)\right)+20+e$ & $(0, \ldots, 0)$ & 0 \\ \hline\xrowht{16pt}
        	Griewank & $1+\sum_{i=1}^d \frac{v_i^2}{4000}-\prod_{i=1}^d \cos \left(\frac{v_i}{i}\right)$ & $(0, \ldots, 0)$ & 0 \\ \hline\xrowht{16pt}
        	Rastrigin & $10 d+\sum_{i=1}^d\left[v_i^2-10 \cos \left(2 \pi v_i\right)\right]$ & $(0, \ldots, 0)$ & 0 \\ \hline\xrowht{16pt}
        	Alpine & $10 \sum_{i=1}^d\big\|\!\left(v_i-v_i^*\right) \sin \left(10\left(v_i-v_i^*\right)\right)-0.1\left(v_i-v_i^*\right)\big\|_2$ & $(0, \ldots, 0)$ & 0 \\ \hline\xrowht{16pt}
        	Salomon & $1-\cos \left(200 \pi \sqrt{\sum_{i=1}^dv_i^2}\right)+10 \sqrt{\sum_{i=1}^d v_i^2}$ & $(0, \ldots, 0)$ & 0 \\
       		\bottomrule
    	\end{tabular}
    	\caption{Benchmark test functions}
    	\label{table:benchmark}
	\end{table}

    \noindent
    \revised{In the subsequent tables we report comparison results for the two methods for the different benchmark functions as well as different numbers of particles $N$ and, potentially, different numbers of steps $K$.
    Throughout, we set $v_b=0$ and $R=\infty$, which is out of convenience. Any sufficiently large but finite choice for $R$ yields identical results.}
				
	The success criterion is defined by achieving the condition $\normm{\frac{1}{N}\sum_{i=1}^N V^i_{\revised{{K,\Delta t}}}-\globmin}\leq0.1$, \revised{which ensures that the algorithm has reached the basin of attraction of the global minimizer.
    The success rate is averaged over $1000$ runs.}

	\paragraph{Isotropic Case.}
        Let $d=15$.
		In the case of isotropic noise, we always set $\lambda=1$, $\sigma=0.3$, $\alpha=10^5$ and \revised{step-size} $\Delta t=0.02$.
        The initial positions $(V_0^i)_{i=1,\dots,N}$ are sampled i.i.d.\@ from $\rho_0=\mathcal{N}(0,I_d)$.
        In Table~\ref{tab:2} we report results comparing the isotropic CBO method with truncation $M=1$ and the original isotropic CBO method \cite{pinnau2017consensus,carrillo2018analytical,fornasier2021consensus} ($M=+\infty$) for the Ackley, Griewank and Salomon function.
        Each algorithm is run for $K=200$ steps.
				\begin{table}[H]\small
					\centering
						\begin{tabular}{c|c|ccccc}
						\toprule
							\multicolumn{7}{c}{\phantom{XXXXXXXXXXXXXXX}Number of steps $K=200$} \\ 
							\midrule \xrowht{10pt} 
							\bf Test function & $M$ & $N=150$ & $N=300$ & $N=600$ & $N=900$ & $N=1200$ \\ 
							\midrule
							\multirow{2}{*}{\text { Ackley }} & 1&0.978 & 0.999 & 1& 1 & 1 \\ 
							&$+\infty$&0.001 & 0.056 & 0.478 & 0.824 & 0.935 \\
							\hline 
							\multirow{2}{*}{\text { Griewank} } &1& 0.060 & 0.188& 0.5013 & 0.671 & 0.791 \\ 
							& $+\infty$&0& 0& 0.010 & 0.013& 0.032\\
							\hline 
							\multirow{2}{*}{Salomon} &1& 0.970 & 1& 1 & 1& 1\\ 
							& $+\infty$&0.005& 0.068& 0.603& 0.909& 0.979\\
							\bottomrule
						\end{tabular}
					\caption{For the $15$-dimensional Ackley and Salomon function, the CBO method with truncation ($M=1$) is able to locate the global minimum using only $N=300$ particles.
					In comparison, even with an larger number of particles (up to $N=1200$), the original CBO method ($M=+\infty$) cannot achieve a flawless success rate.
					In the case of the Griewank function, the original CBO method ($M=+\infty$) exhibits a quite low success rate, even when utilizing $N=1200$ particles. 
					Contrarily, in the same setting, the CBO method with truncation ($M=1$) achieves a success rate of $0.791$.}
					\label{tab:2}
				\end{table}

				\noindent
				Since the benchmark functions Rastrigin and Alpine are more challenging, we use more particles~$N$ and a larger number of steps~$K$, namely $K=200$ and $K=500$.
                We report the results in Table~\ref{tab:3}.
				\begin{table}[H]\small
					\centering
						\begin{tabular}{c|c|ccccc}
						\toprule
							\multicolumn{7}{c}{\phantom{XXXXXXXXXXXXXXX}Number of steps $K=200$} \\ 
							\midrule \xrowht{10pt} 
							\bf Test function & $M$ & $N=300$ & $N=600$ & $N=900$ & $N=1200$ & $N=1500$ \\ 
							\midrule
							\multirow{2}{*}{	\text { Rastrigin }} & 1&0.180 & 0.256 & 0.298& 0.322 & 0.337 \\
							&$+\infty$&0 & 0 & 0.004 & 0.004& 0.007 \\
							\hline
							\multirow{2}{*}{	\text { Alpine }} &1& 0.029 & 0.049& 0.051 & 0.070 & 0.080 \\
							& $+\infty$&0& 0.001& 0.004 & 0.004& 0.004\\
							\midrule
							\multicolumn{7}{c}{\phantom{XXXXXXXXXXXXXXX}Number of steps $K=500$} \\ 
							\midrule \xrowht{10pt} 
							\bf Test function & $M$ & $N=300$ & $N=600$ & $N=900$ & $N=1200$ & $N=1500$ \\ 
							\midrule
							\multirow{2}{*}{	\text { Rastrigin }} & 1&0.213 & 0.265 & 0.316& 0.326 & 0.343 \\
							&$+\infty$&0.001& 0.004 & 0.005 & 0.009 & 0.010 \\
							\hline
							\multirow{2}{*}{	\text { Alpine }} &1& 0.103 & 0.115& 0.147 & 0.165 & 0.173 \\
							& $+\infty$&0.010& 0.015& 0.033 & 0.037& 0.040\\
							\bottomrule
						\end{tabular}
					\caption{For the $15$-dimensional Rastrigin and Alpine function, both algorithms have difficulties in finding the global minimizer. However, the success rates for the  CBO method with truncation ($M=1$) are significantly higher compared to those of the original CBO method ($M=+\infty$).}
					\label{tab:3}
				\end{table}

	\paragraph{Anisotropic Case.}
	Let $d=20$.
    In the case of anisotropic noise, we set $\lambda=1,\sigma=5,\alpha=10^5$ and step-size $\Delta t=0.02$.
    The initial positions of the particles are initialized with $\rho_0=\mathcal{N}(0,100I_d)$.
    In Table~\ref{tab:4} we report results comparing the anisotropic CBO method with truncation $M=1$ and the original anisotropic CBO method \cite{carrillo2019consensus,fornasier2021convergence} ($M=+\infty$) for the Rastrigin, Ackley, Griewank and Salomon function.
    Each algorithm is run for $K=200$ steps.
				\begin{table}[H]\small
					\centering
						\begin{tabular}{c|c|ccccc}
						\toprule
							\multicolumn{7}{c}{\phantom{XXXXXXXXXXXXXXX}Number of steps $K=1000$} \\ 
							\midrule \xrowht{10pt} 
							\bf Test function & $M$ & $N=75$ & $N=150$ & $N=300$ & $N=600$ & $N=900$ \\ 
							\midrule
							\multirow{2}{*}{	\text { Rastrigin }} &1& 0.285 & 0.928&0.990&1&1\\
							& $+\infty$&0.728&0.952&0.993&1&1\\
							\hline
							\multirow{2}{*}{\text { Ackley }} & 1&0.510 & 0.997 & 1& 1 & 1 \\
							&$+\infty$&0.997 & 1 & 1 & 1 & 1 \\
							\hline 
							\multirow{2}{*}{\text { Griewank} } &1& 0.097&0.458 & 0.576& 0.625 & 0.665  \\
							& $+\infty$&0.093&0.101& 0.157& 0.159 & 0.167\\
							\hline 
							\multirow{2}{*}{	\text { Salomon }} &1& 0.010 & 0.434& 0.925 & 0.998& 1\\
							& $+\infty$&0.622& 0.954& 0.970& 0.934& 0.891\\
							\bottomrule
						\end{tabular}
					\caption{For the $20$-dimensional Rastrigin, Ackley and Salomon function, the original anisotropic CBO method ($M=+\infty$) works better than the anisotropic CBO method with truncation ($M=1$), in particular when the particle number~$N$ is small.
                    In the case of the Salomon function, when increasing the number of particle to $N=900$, the success rates of the original anisotropic CBO method ($M=+\infty$) decreases. In the case of the Griewank function, however, we find that the anisotropic CBO method with truncation ($M=+\infty$) works considerably better than the original anisotropic CBO method ($M=1$).}
					\label{tab:4}
				\end{table}
				
				\noindent
				Since the benchmark function Alpine is more challenging and none of the algorithms work in the previous setting,
				we reduce the dimensionality to $d=15$, choose $\sigma=1$, use $\rho_0=\mathcal{N}(0,I_d)$ to initialize, employ more particles and use a larger number of steps $K$, namely $K=200$, $K=500$ and $K=1000$.
                We report the results in Table~\ref{tab:5}.
				\begin{table}[H]\small
					\centering
						\begin{tabular}{c|c|ccccc}
						\toprule
							\multicolumn{7}{c}{\phantom{XXXXXXXXXXXXXXX}Number of steps $K=200$} \\ 
							\midrule \xrowht{10pt} 
							\bf Test function & $M$ & $N=300$ & $N=600$ & $N=900$ & $N=1200$ & $N=1500$ \\ 
							\midrule
							\multirow{2}{*}{	\text { Alpine }} &1& 0 & 0.006& 0.006 & 0.008 & 0.025 \\
							& $+\infty$&0.001& 0.004&0.008 & 0.007& 0.021\\
							\midrule
							\multicolumn{7}{c}{\phantom{XXXXXXXXXXXXXXX}Number of steps $K=500$} \\ 
							\midrule \xrowht{10pt} 
							\bf Test function & $M$ & $N=300$ & $N=600$ & $N=900$ & $N=1200$ & $N=1500$ \\ 
							\midrule
							\multirow{2}{*}{	\text { Alpine }} &1& 0.130 & 0.224& 0.291 & 0.336 & 0.365 \\
							& $+\infty$&0.083& 0.175& 0.250 & 0.292& 0.330\\
							\midrule
							\multicolumn{7}{c}{\phantom{XXXXXXXXXXXXXXX}Number of steps $K=1000$} \\ 
							\midrule \xrowht{10pt} 
							\bf Test function & $M$ & $N=300$ & $N=600$ & $N=900$ & $N=1200$ & $N=1500$ \\ 
							\midrule
							\multirow{2}{*}{	\text { Alpine }} &1& 0.102& 0.198& 0.293 & 0.340 & 0.368 \\
							& $+\infty$&0.097& 0.179& 0.250 & 0.295& 0.331\\
							\bottomrule
						\end{tabular}
					\caption{For the $15$-dimensional Alpine function, the anisotropic CBO method with truncated noise ($M=1$) works better than the original anisotropic CBO method ($M=+\infty$).}
					\label{tab:5}
				\end{table}
					
\section{Conclusions}
\label{sec:conclusions}
				
    In this paper we establish the convergence to a global minimizer of a potentially nonconvex and nonsmooth objective function for a variant of consensus-based optimization (CBO) which incorporates truncated noise.
    We observe that truncating the noise in CBO enhances the well-behavedness of the statistics of the law of the dynamics, which enables enhanced convergence performance and allows in particular for a wider flexibility in choosing the \revised{noise parameter} of the method, \revised{as we observe numerically.}
    For rigorously proving the convergence of the implementable algorithm to the global minimizer of the objective, we follow the route devised in \cite{fornasier2021consensus}.

%%%%%%%%%%%%%%%%%%%%%%%%%%%%%%%%%%%%%%%%%%%%%%%%%%%%%%%%%%%%
\section*{Acknowledgements and Competing Interests}

    This work has been funded by the KAUST Baseline Research Scheme and the German Federal Ministry of Education and Research, and the Bavarian State Ministry for Science and the Arts.
    In addition to this, MF acknowledges the support of the Munich Center for Machine Learning.
    PR acknowledges the support of the Extreme Computing Research Center at KAUST.
    KR acknowledges the support of the Munich Center for Machine Learning and the financial support from the Technical University of Munich -- Institute for Ethics in Artificial Intelligence (IEAI).
    LS acknowledges the support of KAUST Optimization and Machine Learning Lab. 
    LS also thanks the hospitality of the Chair of Applied Numerical Analysis of the Technical University of Munich for discussions that contributed to the finalization of this work.
				
\bibliographystyle{abbrv}
\bibliography{biblio}
				
\end{document}